\documentclass{article}
\usepackage[a4paper, portrait, margin=1.1811in]{geometry}
\usepackage[english]{babel}
\usepackage[utf8]{inputenc}
\usepackage[T1]{fontenc}
\usepackage{amsfonts, amsmath, amsthm, amssymb, bbm, mathrsfs, esint, dsfont, mathenv, relsize, stmaryrd}
\usepackage{mathtools}
\usepackage[normalem]{ulem}
\usepackage{enumerate, enumitem,multicol}
\usepackage{algorithm}
\usepackage{algpseudocode}
\usepackage{helvet}
\usepackage{etoolbox}
\usepackage{graphicx}
\usepackage{titlesec}
\usepackage{caption}
\usepackage{booktabs}
\usepackage[dvipsnames]{xcolor} 
\usepackage{csquotes}

\usepackage{setspace}
\graphicspath{ {./images/} }
\usepackage{scrextend}
\usepackage{fancyhdr}
\usepackage{graphicx}
\usepackage{subcaption}
\usepackage[backend=bibtex, natbib=true, style=authoryear, giveninits=true, maxbibnames=99, doi=false,isbn=false,url=false,eprint=false,uniquename=init]{biblatex}
\RequirePackage[colorlinks=true,citecolor=blue!50!black,linkcolor=blue!50!black,urlcolor=blue!50!black,pdfauthor=author]{hyperref}
\usepackage[capitalize,nameinlink,noabbrev]{cleveref}

\theoremstyle{plain}
\newtheorem{theorem}{Theorem}[section] % le [chapter] sert à numéroter par chapitre 
\newtheorem{proposition}[theorem]{Proposition}
\newtheorem{definition}[theorem]{Definition}
\newtheorem{coro}[theorem]{Corollary}
\newtheorem{lem}[theorem]{Lemma}
\theoremstyle{remark}

\crefname{thm}{Theorem}{Theorems}
\crefname{lem}{Lemma}{Lemmas}
\crefname{prop}{Proposition}{Propositions}
\crefname{cor}{Corollary}{Corollaries}

\def\EE{\mathbb{E}}

\def\RR{\mathbb{R}}

\def\1{\mathds{1}}
\def\H{\mathcal{H}}%%%la lettre H en ronde majuscule
\renewcommand{\hat}{\widehat}
\renewcommand{\tilde}{\widetilde}

\newlist{hypDiamond}{enumerate}{1} % Avec des Diamants
\setlist[hypDiamond,1]{label={$\diamond$}}

\makeatletter   
\patchcmd{\@maketitle}{\LARGE \@title}{\fontsize{16}{19.2}\selectfont\@title}{}{}
\makeatother

\makeatletter
\def\namedlabel#1#2{\begingroup
   \def\@currentlabel{#2}%
   \label{#1}\endgroup
}
\makeatother

\usepackage{authblk}

\newsavebox\affbox
\author[1]{\textbf{Luca Castelli}}
\affil[1]{Univ Lyon, Université Claude Bernard Lyon 1, CNRS UMR 5208, Institut Camille Jordan, F-69622 Villeurbanne, France\protect\\~ } 
\titleformat{\section}{\normalfont\fontsize{10}{15}\bfseries}{\thesection.}{1em}{}
\titleformat{\subsection}{\normalfont\fontsize{10}{15}\bfseries}{\thesubsection.}{1em}{}
\titleformat{\subsubsection}{\normalfont\fontsize{10}{15}\bfseries}{\thesubsubsection.}{1em}{}

\titleformat{\author}{\normalfont\fontsize{10}{15}\bfseries}{\thesection}{1em}{}

\title{\textbf{\huge Estimation of projection operators with Gaussian noise}\\}
\date{\today}    
\usepackage{todonotes}

\begin{document}

\pagestyle{headings}	
\newpage
\setcounter{page}{1}
\renewcommand{\thepage}{\arabic{page}}

\setlength{\parskip}{12pt} % saut après paragraphe
\setlength{\parindent}{0pt} % identation en début de paragraphe
\onehalfspacing  % \setstretch{1.3} % espace entre les lignes
	
\maketitle
	
\noindent\rule{15cm}{0.5pt}
	\begin{abstract}
	This paper focuses on random projection operators when the subspace of projection is estimated. We derive non-asymptotic upper bounds on the error between the projection onto the estimated subspace and the projection onto the underlying subspace. The provided upper bounds depend on the noise and on intrinsic properties of the estimated subspace. Several scenarios are considered according to the distribution of the estimator of the matrix spanning the subspace. The aforementioned bounds are attained under a structural assumption on the Gram matrix associated with the subspace. Regularized estimators are introduced to circumvent this assumption. An example is given in the partial least square (PLS) framework where the estimated subspace is spanned by the PLS weights.
	\end{abstract}
\noindent\rule{15cm}{0.4pt}
\section{Introduction}\label{Sec:Intro}
In a high-dimensional setting, determining the intrinsic structure of data represents a major challenge. One way to extract useful information is to approximate a low-dimensional subspace thanks to dimension reduction. For this purpose, methods such as principal component analysis (PCA), principal component regression (PCR) or partial least squares (PLS) are widely used. In several models the methods aim to estimate a specific subspace that often satisfies a predefined constraint. A natural question that arises is how to assess the quality of this estimation. We are motivated to approximate a subspace $\H=[H]$, associated with a matrix $H\in\RR^{n\times K}.$ Formally, we aim to estimate this subspace by considering the random subspace $\hat{\H}=[\hat{H}]$ with $\hat{H}=H+E$.

Given a matrix $A \in \mathbb{R}^{p\times s}$, $[A]:=\mathrm{span}(A)$ denotes the subspace of $\mathbb{R}^p$ generated by the columns of$~A$. If $A \in \mathbb{R}^{s\times s}$ is a positive definite matrix, the highest and the lowest eigenvalues will be denoted respectively by $\rho(A)$ and $\rho_{\min}(A)$, its trace by $\mathrm{Tr}(A)$, while its condition number writes $\mathrm{Cond}(A)$. The diagonal matrix $\mathrm{diag}(A_{11},\dots, A_{ss})$ extracted from $A$ will be written $\mathrm{diag}(A)$. The $\ell^2$ norm on $\RR^{p}$ (or $\RR^{n}$) is written $\|.\|$.

Assuming that $\mathcal{H}$ and $\hat{\mathcal{H}}$ have dimension $K$, to measure the difference between the two spaces, one can consider the distance between the projection operators. 
\begin{equation}\label{Eq:Dist}
\mathrm{dist}(\H,\hat{\H})=\frac{1}{\sqrt{n}}|||P_{[H]}-P_{[\hat{H}]}|||,
\end{equation}
where $|||\cdot|||$ is the operator norm. We refer to \textcite{BAKSALARY} for an overview on distances and angles between subspaces.
\cref{Eq:Dist} allows subspace estimation to be reformulated as the estimation of a projection operator.
Based on this consideration we aim to estimate a projection operator $P_{[H]}:\mathbb{R}^{n}\to \mathbb{R}^{n}$ based on an estimation $\hat{H}$ of the matrix $H$, where we evaluate the estimation error between $P_{[H]}$ and $P_{[\hat{H}]}$. This is done by providing a non-asymptotic bound on the distance 
$$\frac{1}{n}|||P_{[H]}-P_{[\hat{H}]}|||^{2}.$$
This work has a major application in Partial Least Square regression (\textcite{Wold1983}). This method aims to approximate a specific subspace by projecting a response vector onto the latter (we refer to \cref{Subsec:LinearModel} for an introduction). By construction, the space considered for PLS regression is complex to compute due to its nonlinearity and the dependencies present among the columns of the matrix $H$ (see \textcite{Kramer2007} for more details). Consequently, we alleviate the challenges encountered in the study of PLS regression by considering several simple and general frameworks for the estimation of the subspace. This translates into independence assumptions on the distribution of the estimated matrix $\hat{H}=H+E$. First, we focus on a perturbed version of $H$ where each entries of $E$ are i.i.d and Gaussian. This very general case can be used to model the approximation of an arbitrary space or to represent a space that has been affected by noise. The estimation of the projection operator is linked to the inverse of the matrix $H^TH$ which results in a constraint on the eigenvalues of this matrix. Under some specific assumption on the smallest eigenvalue of the matrix $H^TH$ we establish that, with high probability,
$$\frac{1}{n}|||P_{[H]}-P_{[\hat{H}]}|||^{2}\le C\ \mathrm{Cond}(H^TH)^{3}\frac{\gamma^{2}}{\rho_{\min}(H^TH)},$$
where $C$ is a positive constant and $\gamma$ characterizes the noise level.
The provided non-asymptotic bound is dependent on a ratio between the level of noise $\gamma^{2}$ and the smallest eigenvalue $\rho_{\min}(H^TH)$ and on the choice of the basis $H$ with the condition number $\mathrm{Cond}(H^TH)$ of the Gram matrix$~H^TH$. This highlights the importance of the choice of basis through the matrix $H$.
A general approach is then developed for  more complex relations between the rows and columns of the estimated matrix$~\hat{H}$. Consequently, we consider different scenarios according to the assumptions on the matrix$~E$. Our main contribution is an upper bound for the scenario where we assume that each row follows a given multivariate normal distribution. This scenario is similar to considering a random design where each row has a specific distribution, and it can be applied in various contexts. This contribution is followed by a general scenario where we consider a vector which spans the estimated subspace thanks to a given family of matrices. This case can be seen as a generalization of the Krylov space where the structure of the matrices is less restrictive. It will be applied to the PLS framework for the estimation of the Krylov subspace spanned by the PLS weights (see \cref{Subsec:ProjOperatorPLSCase}). 

Then regularized estimators are introduced to remove the aforementioned assumption thanks to a Ridge regularization on the estimated Gram matrix. The given upper bounds for the regularized estimators are similar to those obtained under the assumption on $\rho_{\min}(H^TH)$.
\paragraph*{}The paper is structured as follows. \cref{Sec:ProjIntro} is a detailed introduction that motivates our approach. We discuss how random projection operators have already been studied in multiple contexts such as regression or PCA and we present an overview of some results associated to this concept. \cref{Subsec:ProjOperatorFramework} states the general framework and the different scenarios which will be considered with a summary of the contributions of the paper. In \cref{Sec:ProjOperatorMain}, some of the upper bounds according to two specific scenarios are presented, followed by the results obtained with their regularized versions. These have been simplified for the ease of exposition, more precise results have been proved in the appendix. Finally the PLS case is investigated with this formalism. The upper bounds related to the other scenarios are displayed in the appendix.
\section{Motivation}\label{Sec:ProjIntro}
The curse of dimensionality can be handled through dimensionality reduction by projecting the data onto a more general $K$-dimensional subspace. This principle is the foundation of PCR and PLS regression. We will highlight how projection operators appear in these methods in the linear model, as well as in other more general methods, discussing their differences and specific characteristics. 
\subsection{Linear model and projection}\label{Subsec:LinearModel}
Our investigations are highly motivated by the prediction task in linear model that involves projection.
Let us introduce the classical regression model. We observe a $n$-sample $(X_i,Y_i)$, $i=1,\dots,n$, where the $Y_i\in\RR$ are outcome variables and the $X_i\in\RR^p$ $p$-dimensional covariates. We consider a linear model with each couple $(X_i,Y_i)$, represented by the equation
\begin{equation}
\label{Eq:Modele lineaire1}
Y=X\beta +\varepsilon,
\end{equation}
where $\varepsilon=(\varepsilon_1,\dots \varepsilon_n)^T\sim\mathcal{N}\big(0,\tau^2I_{n}\big)$, $X=(X_1,\dots, X_n)^T\in \RR^{n\times p}$ and $Y=(Y_1,\dots,Y_n)^T$ $\in \RR^n$. Here and below, the matrix $I_{n}$ is the identity matrix of size $n$, the parameter $\tau >0$ characterizes the noise level and the exponent $T$ denotes the transpose operator. The design matrix $X$ is considered as deterministic. The associated Gram matrix is written $\Sigma=\frac{1}{n} X^TX$.  We denote by $\hat{\sigma}=\frac{1}{n}X^TY$ the normalized scalar product between $X$ and the response vector $Y$. The so-called population version of this last quantity is written $\sigma=\mathbb{E}(\hat\sigma)$ where $\mathbb{E}$ denotes the expectation.

\paragraph*{} Depending on the objective, the notion of projection can be used in different ways for the linear model. A first approach consists in considering situations where the variables are subject to errors (see \textcite{Fekri} and \textcite{Gillard}). This amounts to considering $\hat{X}=X+E$ and then applying the OLS estimator on the design $\hat{X}$. Hence, the term can be seen as $$\hat{X}\hat{\beta}_{OLS}=P_{[\hat{X}]}(Y).$$ Studying this model therefore amounts to considering the difference between the spaces $[X]$ and$~[\hat{X}]$. We refer to \textcite[Chapter 1 and 2]{Fuller} for further informations.

In a high dimensional context, namely when $p$ is larger of the number of observations $n$, the usual OLS estimator is not defined. There are however multiple ways to provide an estimator for the parameter $\beta$ or the prediction of $X\beta$ to circumvent this fact (see \textcite{Giraud} for a complete introduction). In particular, one approach consists in reducing the number of covariates by combining them linearly; this is the main idea of dimension reduction in regression. In this paradigm the estimation of $\beta$ includes a constraint $\beta\in[W]$ where $[W]$ is a $K-$dimensional subspace, $K<p$. We obtain the following estimator,
$$\hat{\beta}_{K,W}=\underset{u\in[W]}{\mathrm{argmin}}\ \frac{1}{n}\|Y-Xu\|^{2}.$$
One can show that $$X\hat{\beta}_{K,W}=P_{[XW]}(Y),$$ which illustrates how constraining the parameter $\beta$ to lie in a specific subspace is linked to projection operator. Thus, in the case where $W$ is estimated, one can consider $\hat{W}$ computed as an estimator of $W$. Providing inference on the estimation of the projection on the subspace $H=XW$ can be related to the framework of dimension reduction for regression by considering $\hat{H}=X\hat{W}$ through the computation of the quadratic loss in prediction.

We can write for $\hat{H}=\hat{X}$ (error in variables) or $\hat{H}=X\hat{W}$ (dimension reduction). Let $\hat{\beta}_{\hat{H}}$ being the OLS estimator on $[\hat{X}]$ for error in variables and being equal to $\hat{\beta}_{X\hat{W}}$ for dimension reduction.
\begin{equation}\label{Eq:ModelLinearDecomp}
\frac{1}{n}\|X(\hat{\beta}_{\hat{H}}-\beta)\|^{2}\le \frac{2}{n}\|P_{[\hat{H}]}(Y)-P_{[H]}(X\beta)\|^{2}+\frac{2}{n}\|X\beta-P_{[H]}(X\beta)\|^{2},
\end{equation}
where the second term in the right-hand side can be interpreted as a bias term. However, the first term can be related to projection operators with the following inequality
\begin{align*}
    \frac{2}{n}\|P_{[\hat{H}]}(Y)-X\beta\|^{2}&\le \frac{4}{n}\|\big(P_{[\hat{H}]}-P_{[H]}\big)(X\beta)\|^{2}+\frac{4}{n}\|P_{[\hat{H}]}(\varepsilon)\|^{2}.
\end{align*}
The first term in the right-hand side can be controlled by the norm $|||P_{[H]}-P_{[\hat{H}]}|||$. This inequality thus establishes a connection between the quadratic loss in prediction and the estimation of projection operators. We now describe how, in the case of dimension reduction, the matrix W can be computed.

The matrix $W$ is often designed to satisfy a specific optimization problem which depends on the objective under consideration. The optimization problem can be related to the covariates or the response variable, this leads to two different methods. 

In PCR regression, the matrix $W$ is composed of the eigenvectors of the matrix $\Sigma$.
The matrix $W$ corresponds to $XV$ where $V$ is the matrix having the eigenvectors of the Gram matrix $X^TX$ as columns vectors associated with the $K$ largest eigenvalues (we refer to \textcite{ESLI} and \textcite{Giraud} for a comprehensive introduction). The subspace $[XV]$ is not estimated in this scenario because $X$ is assumed to be deterministic. This case can be explained by the optimization problem satisfied by $V$, which consists of maximizing the available information in the matrix $X$.

In PLS regression, the optimization problem also involves the response $Y$.
This paper was first motivated by the estimation of the matrix $H=XW$ for the PLS case, relying on the construction of the PLS estimator, which is intrinsically linked to this notion. The PLS method originated first in economy (\textcite{Wold1966}) to compute principal components throught an iterative process named \textit{Non-linear Iterative Partial Least Squares} (NIPALS) (\textcite{Wold1973}). The PLS algorithm appeared in \textcite{Wold1983}. The main difference with PCR regression lies in the choice of the subspace computed by the algorithm. The subspace is computed to obtain components which are the most correlated with the vector $Y$. We refer to \textcite{Aparicio} for a complete introduction on PLS regression. The PLS algorithm gives an orthogonal basis of this subspace denoted by the matrix $W:=(w_{1},...,w_{K})$ whose columns vectors are referred to as weights.  The iterative form of the PLS algorithm makes the statistical analysis of the method difficult. \textcite{Helland}
demonstrated that $[W]=\hat{\mathcal{G}}$, where $\hat{\mathcal{G}}$ denotes the Krylov space defined as 
$$\hat{\mathcal{G}}:=[\hat{G}]\quad \text{with}\quad \hat{G}=(\hat{\sigma},\Sigma\hat{\sigma},...,\Sigma^{K-1}\hat{\sigma}).$$
This choice of basis allows for deriving an explicit formula for both the PLS estimator and a basis of the subspace $\hat{\mathcal{G}}$. 
Considering the Krylov spaces spanned by the column vectors of $G=(\sigma,\Sigma\sigma,...,\Sigma^{K-1}\sigma)$ and its sample version $\hat{G}=(\hat{\sigma},\Sigma\hat{\sigma},...,\Sigma^{K-1}\hat{\sigma})$, we observe that PLS regression aims to approximate the subspace $[XG]$ with $[X\hat{G}]$ and to project the response variable onto the latter. These two subspaces have a particular structure that makes the estimated vectors highly dependent on each other.
For PLS regression, this results in an estimation $\hat{\mathcal{H}}=[X\hat{G}]$ of the subspace $\mathcal{H}=[XG].$ Therefore, this makes the projection subspace $\hat{\mathcal{H}}$ random.
The framework of these Krylov spaces may already seem complex due to the dependence that can be found in both rows and columns of the matrix $X\hat{G}$. The matrix $X\hat{G}$ can be written 
$$X\hat{G}=XG+E,$$
where $E$ is centered and has dependent Gaussian column vectors.
Our work in this paper builds on this idea by considering simple cases where we relax dependency constraints to better study and understand the phenomenon. This approach allows us to increase the constraints on the matrices in question and recover the case of Krylov subspaces later.

\subsection{Estimation of Projection Operators} Stepping beyond the scope of this paper, we point to ongoing work on the estimation of projection operators.
We have described projection operators in a regression context, where dimension reduction is done thanks to a deterministic or random subspace with a specific distribution.
Moving beyond this regression framework, one could be inclined to focus on estimating a particular given projection operator. Spectral operators, for instance, can be used to estimate eigensubspaces. In our framework this amounts to set $H=V$ where $V$ corresponds to the eigendecomposition of the operator and $\hat{H}$ its estimation. The estimation of the Gram and covariance matrices has consistently been related to such problems through the spectrum and eigenvalues.
%As we saw for PCR and PLS regressions the spectrum of $\Sigma$ is playing a central role, whether in the prediction bounds obtained (see \textcite{Cast} and \textcite{Cast2}), the interpretation of the data (ACP), or the theoretical coordinates of our parameter $\beta$ (see \textcite{Cast2}).%
Hence estimation of the eigendecomposition of the covariance matrix induces the study of a spectral random operator. 

In this section we displayed some contributions related to estimation of projection operator on eigensubspaces. We can mention for instance \textcite{KoltchinskiiLounici2017} (see also \textcite{KoltchinskiiLounici2016PCA} for discussion on PCA) where the covariance matrix $\Sigma$ is estimated by its sample version $\hat{\Sigma}$ in  a Hilbert space $\mathcal{H}$. The sample covariance $\hat{\Sigma}$ is computed thanks to $n$ independent copies $X_{i}$ of a centered random variable $X$. The authors first provide a bound on the operator norm $|||\hat{\Sigma}-\Sigma|||$ with high probability. %The bound is of the form $|||\Sigma|||\max\left(\frac{\mathbf{r(\Sigma)}}{n},\sqrt{\frac{\mathbf{r}(\Sigma)}{n}}\right)$, up to a constant depending on the probability of the event.%
They introduced the quantity $\mathbf{r}(\Sigma)=\frac{\mathrm{Tr}(\Sigma)}{|||\Sigma|||}$, called the effective rank of the matrix $\Sigma$, to focus on projection operators on eigensubspaces. The effective rank appears in concentration inequalities for sample covariance operators (\textcite{KoltchinskiiLounici2017Gram}). It has also been used for concentration bounds on bilinear forms of spectral projectors, e.g in \textcite{KoltchinskiiLounici2016}. The authors take into account the spectral projector $P_{r}$ on the eigensubpace related to the $r$-th eigenvalue. Hence this operator is estimated by the sample version $\hat{P}_{r}$. The authors derive a concentration bound on the Hilbert-Schmidt norm $\|\hat{P}_{r}-P_{r}\|_{\mathcal{HS}}$ and around its expectation. The provided bound is given with high probability under specific assumptions on the covariance matrix $\Sigma$. The bound, up to a constant $C$, is the following:
$$|\|P_{r}-\hat{P}_{r}\|_{\mathcal{HS}}-\mathbb{E}[\|P_{r}-\hat{P}_{r}\|_{\mathcal{HS}}]|\le C\cdot \max\left(\frac{B_{r}(\Sigma)}{n},\frac{|||\Sigma|||^{2}}{ng_{r}^{2}},\frac{|||\Sigma|||^{3}}{g_{r}^{3}}\frac{\mathbf{r}(\Sigma)}{n\sqrt{n}}\right),$$ where $B_{r}(\Sigma)=|||P_{r}\Sigma P_{r}|||\cdot|||C_{r}\Sigma C_{r}|||$, $C_{r}=\sum_{s\neq r}\frac{1}{\mu_{r}-\mu_{s}}P_{s}$ where $\mu_{s}$ denotes distinct nonzero eigenvalues of the matrix $\Sigma$ arranged in decreasing order. The term $g_{r}$ is called the $r$-th spectral gap and is equal to $\mu_{r}-\mu_{r+1}>0,$ the terms $|||\Sigma|||, g_{r}$ and $\mathbf{r}(\Sigma)$ show how the eigendecomposition of the matrix is involved in the estimation of the spectral estimator $P_{r}.$ The spectral gap $g_{r}$  reflects the way of distinguishing an eigenvalue from other eigenvalues in order to separate the information, as two close eigenvalues make the estimation of the associated spectral projectors more difficult. 

The connection between their work and ours is established through the spectrum of the covariance matrices involved in the resulting bounds (we refer to \cref{Sec:ProjOperatorMain}, \cref{Th:ProjOperatorRowsI}).

In the context of Hilbert space $\mathcal{H}$ of dimension $p\in\mathbb{N}\cup\{+\infty\}$, \textcite{ReissWahl} also use an approach based on spectral projectors in order to obtain non-asymptotic upper bounds on the excess risk for principal component analysis. The authors consider the projector $P_{\le d}$ corresponding to the orthogonal projections of the first $d$-eigenvalues arranged by decreasing order of the matrix $\Sigma$ and its estimated version $\hat{P}_{\le d}$. This projection operator satisfies 
$$P_{\le d}\in\underset{P\in\mathcal{P}_{d}}{\mathrm{argmin}}\ \mathbb{E}[\|X-PX\|^{2}],$$
with  $\mathcal{P}_{d}=\{P : \mathcal{H}\to\mathcal{H}\ |\ P\text{ is an orthogonal projection of rank }d\}$. Denote $R(P):=\mathbb{E}[\|X-PX\|^{2}]$ the reconstruction error. The excess risk of the projector $\hat{P}_{\le d}$ is defined as 
$$\mathcal{E}_{d}^{PCA}:=R(\hat{P}_{\le d})-R(P_{\le d}).$$
In \textcite{ReissWahl}, the relation is highlighted thanks to a control of the Hilbert-Schmidt distance\ by the excess risk through the inequalities
\begin{equation}\label{Eq:ReissWahlSpectralProj}
\frac{2\mathcal{E}_{\le d}^{PCA}(\lambda_{d+1})}{\lambda_{1}-\lambda_{d}}\le\|\hat{P}_{\le d}-P_{\le d}\|_{\mathcal{HS}}^{2} \le\frac{2\mathcal{E}_{\le d}^{PCA}(\lambda_{d+1})}{\lambda_{d}-\lambda_{d+1}} \le \frac{2\mathcal{E}_{\le d}^{PCA}}{\lambda_{d}-\lambda_{d+1}},
\end{equation}
where $\mathcal{E}_{\le d}^{PCA}(\mu)=\displaystyle\sum_{j\le d}(\lambda_{j}-\mu)\|P_{j}\hat{P}_{>d}\|_{\mathcal{HS}}^{2}$ is a representation of the excess risk with $\hat{P}_{>d}=I-\hat{P}_{\le d}$ with $I$ the identity operator. Inequality \eqref{Eq:ReissWahlSpectralProj} shows that the excess risk bounds the Hilbert-Schmidt distance up to the spectral gap.
The main objective of \textcite{ReissWahl} is to give a non-asymptotic upper bound on 
$\mathbb{E}[\mathcal{E}_{d}^{PCA}]$.
The upper bound is depending on the spectral gap $\lambda_{j}-\lambda_{d}$ discussed previously. The bound involves the spectral elements of the covariance operator $\Sigma$ and aligns with the terms induced in the results discussed previously from \textcite{KoltchinskiiLounici2017}. The displayed bound can be adapted to Kernel PCA or functional PCA thanks to the Hilbert point of view. We refer to \textcite{ReissWahl} for a complete discussion on spectral gap and several improvements of this upper bound according to different scenarios.

\subsection{Other related methods and questions}
We have motivated the interest of the estimation of projection operator for regression in linear model and also for the specific case of eigenspaces. Moreover, one can cite other methods based on the notion of random projection that do not fall within the scope of our framework (see \cref{Sec:Intro} and \cref{Subsec:ProjOperatorFramework}). Returning to the linear model introduced in \cref{Subsec:LinearModel}, we can cite two different methods related to projections, which are based on a multiplication of the covariate matrix $X$ by a random matrix. Such approaches can be developed without solving specific optimization problems, but simply by considering a randomly generated subspace. methods are interesting but fall outside our model; however, they involve projection operators and are therefore relevant. They will not be discussed further beyond this section.
In the following, let $W\in\mathbb{R}^{K\times n}$ be a random matrix with independent entries, zero mean and unit variance. 
\paragraph*{Sketching.} Sketching is a method where the parameter $K$ needs to satisfies $n>K>p$ in order to keep more observations than the number of covariates. It consists in replacing the optimization problem $\underset{v\in\mathbb{R}^{p}}{\mathrm{argmin}}\|Y-Xv\|^{2}$ by $$\underset{v\in\mathbb{R}^{p}}{\mathrm{argmin}}\|WY-WXv\|^{2}.$$ 
It can be observed that Sketching performs a direct change of dimension, reducing from
$n\times p$ to $K\times p$. The idea of Sketching is to replicate $m\in \mathbb{N}$ versions of $(W^{(j)})_{j=1,...,m}$ of $W$, then to compute $\hat{\beta}_{S}^{(j)}$ (with respect to $W^{(j)}$ for $j=1,...,m$) minimizing the new optimization problem induced by $W^{(j)}$. The last step is to compute the Sketching estimator as $\hat{\beta}_{S}=\frac{1}{m}\displaystyle\sum_{j=1}^{m}\hat{\beta}_{S}^{(j)}.$ 
From this perspective this amounts to considering $\hat{H}=WX$. For $j\in\{1,...,m\}$, the estimator $\hat{\beta}_{S}^{(j)}$ can be associated with a random projection, as outlined by the following expression: 
$$W^{(j)}X\hat{\beta}^{(j)}_{S}=P_{[W^{(j)}X]}(W^{(j)}Y).$$
The estimator $\hat{\beta}_{S}$ is obtained by averaging the estimators constructed from each replication of $W$, which are determined through projections onto $\hat{H}$. Here, the random projections are present for the estimation of each $\hat{\beta}_{S}^{(j)}$, but there is no theoretical subspace $H$ being estimated.
We refer to \textcite[Chapter 10]{Bach} for a displayed upper bound. The advantage of this method is primarily computational. We refer to \textcite{Dobriban} for more details on Sketching.

\paragraph*{Random projections.}Another way to proceed in a high dimensional context when the dimension $K$ satisfying $p>n>K$ is to consider $W\in\mathbb{R}^{p\times K}$ with the optimization problem
$\underset{v\in\mathbb{R}^{p}}{\mathrm{argmin}}\|Y-Xv\|^{2}$ replaced by $$W\cdot\underset{v\in\mathbb{R}^{K}}{\mathrm{argmin}}\|Y-XWv\|^{2}.$$
As for the Sketching, we consider $m\in\mathbb{N}$ replications $W^{(j)}_{j=1..m}$ of $W$. For each $j=1,...,m$, we consider the estimator $\hat{\beta}_{RP}^{(j)}$ related to the optimization problem with $W^{(j)}$. The global estimator is computed as $\hat{\beta}_{RP}=\frac{1}{m}\displaystyle\sum_{j=1}^{m}\hat{\beta}_{RP}^{(j)}$. The difference with Sketching lies in the dimension. We have a dimension reduction  model where the subspace $W$ is random and not chosen to satisfy a specific optimization problem (like PCR or PLS). We recall that $W$ has a Gaussian distribution. Each estimator $\hat{\beta}_{RP}^{(j)}$ is obtained thanks to random projections $P_{[\hat{H}]}$ where $\hat{H}=XW$. We can highlight the relation with the projection operator thanks to the identity
$$X\hat{\beta}^{(j)}_{RP}=P_{[XW^{(j)}]}(Y).$$
For $j\in\{1,...,m\}$ we have $X\hat{\beta}_{RP}^{(j)}=\Pi^{(j)} Y,$
where $\Pi^{(j)}$ is the matrix of the projection operator $P_{[XW^{(j)}]}$. Let $\Theta_{\Pi}=\mathbb{E}[\Pi^{(j)}]$, the eigenvalues of this matrix are in $(0,1)$ and can be bounded in a Gaussian context. We refer to \textcite[Chapter 10]{Bach} for an upper bound and a complete discussion. This approach can be applied for kernel methods with supplementary discussions on features represented by Gaussian random projection representations (see \textcite{JohnsonWilliam} for a detailed inequality between features and their Gaussian random projection representation).

\section{General framework}\label{Subsec:ProjOperatorFramework}
In \cref{Sec:ProjIntro}, we have introduced several framework for the estimation of projection operators in various contexts. We now focus on the signal-plus-noise model described in \cref{Sec:Intro}.  

Let $H\in\mathbb{R}^{n\times K}$ having full rank $K$, and $\hat{H}=H+E$ with $E\in\mathbb{R}^{n\times K}$ an error term. Our aim is to provide non-asymptotic bounds for the estimation error 
$$\frac{1}{n}|||P_{[\hat{H}]}-P_{[H]}|||^{2}=\frac{1}{n}|||H(H^TH)^{-1}H^T-\hat{H}(\hat{H}^T\hat{H})^{-1}\hat{H}|||^{2},$$
under a minimal set of assumptions. The  formula of $P_{[\hat{H}]}$ includes the estimate of the inverse of the matrix $H^TH$, which results in the importance of the spectrum of the matrix $H^TH$. We consider four main scenarios which are associated to different distributions on $E$.
\paragraph*{Scenario 1: Independent rows and columns} We assume that the error terms $E_{ij}$ are i.i.d and follow a normal distribution. Hence we assume for all $i\in\{1,...,n\}$ and $j\in\{1,...,K\}$ $$E_{ij}\sim\mathcal{N}(0,\gamma^{2}).$$
Scenario 1 can be seen as a first step to understand the underlying mechanism between $H$ and $\hat{H}$. One example can be to consider the PCA on a design matrix $X$ where $[H]$ is the subspace spanned by $K$ eigenvectors and $\hat{H}$ a noisy version of each principal component. Another direct application lies in the linear model as detailed in \cref{Subsec:LinearModel} when one allows for error in the variables in the matrix $X$ by considering $\hat{X}=X+E$. See \textcite{Fekri} on errors in variables and \textcite{Gillard} for a complete introduction.
This scenario would allow us to quantify the error between the subspaces $[X]$ and~$[\hat{X}]$ through the distance between the projection operators.
\paragraph*{Scenario 2: Independent rows} We add dependence between the "covariates" thanks to a covariance matrix $S$, up to a noise factor. We assume for all $i\in\{1,...,n\}$ $$E_{i,\cdot}^{T}\sim\mathcal{N}_{K}(0,\gamma^{2}S), \ \text{with}\ S\in\mathbb{R}^{K\times K}.$$
Scenario 2 is a generalized version of Scenario 1 where the noise has dependence for each row. Under these assumptions, the covariates are linked through the matrix $S$. This scenario may be interpreted as the subspace being disturbed, leading to noise-induced dependencies across the columns. One can consider the same case as in Scenario 1 with the linear model that allows for errors in the variables with correlated noise.
\paragraph*{Scenario 3: Independent columns} We add dependence in "individuals" by assuming independence between the columns of $E$. We assume for all $j\in\{1,...,K\}$
$$E_{\cdot,j}\sim\mathcal{N}_{n}(0,\gamma^{2}A), \ \text{with} \ A\in\mathbb{R}^{n\times n}.$$
Scenario 3 can be seen as the symmetric counterpart of Scenario 2 where the dependence in the noise in now in the rows. As the second scenario has been described, we do not discuss this case further.
\paragraph*{Scenario 4: Generalized basis} The last scenario is a generalized version of estimated Krylov subspaces where the columns and rows are dependent. We consider a non zero vector $v\in \RR^{n}$.
Let for $j\in\{1,...,K\}$, such that $\{(A_{j}v)_{j\in\{1,...,K\}}\}$ is a basis of $[H]$ with $A_{j}\in\mathbb{R}^{n\times n}$.
We consider an estimator $\hat{v}$ such that $\hat{v}\sim\mathcal{N}(v,\gamma^{2}V)$.
then we consider $\hat{H}$ such that $\hat{H}_{\cdot,j}=A_{j}\hat{v}$. Hence $\hat{H}=H+E$ where for $j\in\{1,...,K\},$ $$E_{\cdot,j}\sim\mathcal{N}_{n}(0,\gamma^{2}A_{j}VA_{j}^{T}).$$
Remark that the rows and columns are not independent. This corresponds to the context of PLS regression, the dependence is typically complex and closer to this scenario.

\paragraph*{Summary of the contributions} We summarize all the contributions of this paper in the following table, where $H\in\mathbb{R}^{n\times K}$ with $\hat{H}=H+E$. Bounds on $\frac{1}{n}|||P_{[H]}-P_{[\hat{H}]}|||^{2}$ given in \cref{Tab:Summary} are stated with high probability. It also gives the assumptions under which the bounds are obtained.
\begin{table}[H]
\centering
\scalebox{0.89}{
\begin{tabular}{|c|c|c|c|}
\hline
\textbf{Scenario}&\textbf{Distribution} & \textbf{Assumption} & \textbf{Bound} \\
\hline
 1&$E_{ij}\sim\mathcal{N}(0,\gamma^{2})$ & $\rho_{\min}(H^TH)\ge D_{\delta} \gamma^{2}n$ & $c_{\delta,H}\dfrac{\gamma^{2}}{\rho_{\min}(H^TH)}$\\[0.3cm]
\hline
2&$E_{i,\cdot}^{T}\sim\mathcal{N}_{K}(0,\gamma^{2}S)$&$\rho_{\min}(H^TH)\ge D_{\delta}\gamma^{2}n\rho(S)$ & $c_{\delta,H}\gamma^{2}\dfrac{\rho(S)}{\rho_{\min}(H^TH)}$\\[0.3cm]
\hline
3&$E_{\cdot,j}\sim\mathcal{N}_{n}(0,\gamma^{2}A)$& $\displaystyle\rho_{\min}(H^TH)\ge D_{\delta}\gamma^{2}n\rho(A)$& $c_{\delta,H}\gamma^{2}\dfrac{\rho(A)}{\rho_{\min}(H^TH)}$\\[0.3cm]
\hline
4&$E_{\cdot ,j} \sim \mathcal{N}_{n}(0,\gamma^2A_jVA_j^T)$&$\displaystyle\rho_{\min}(H^TH)\ge D_{\delta,K}\gamma^{2}\sum_{j=1}^{K}\mathrm{Tr}(A_{j}VA_{j}^{T})$& $\displaystyle c_{H}D_{\delta,K}\frac{\gamma^{2}}{n}\sum_{j=1}^{K}\frac{\mathrm{Tr}(A_{i}VA_{i}^{T})}{\rho_{\min}(H^TH)}$\\
\hline
4&PLS & $\displaystyle\rho_{\min}(\Theta)\ge D_{\delta,K}\frac{\tau^{2}}{n}\sum_{j=1}^{K}\mathrm{Tr}(\Sigma^{2i})$&$c_{\Theta}D_{\delta,K}\dfrac{\tau^{2}}{n}\displaystyle\sum_{j=1}^{K}\frac{\mathrm{Tr}(\Sigma^{2j})}{n\rho_{\min}(\Theta)}$\\[0.3cm]
\hline
\end{tabular}}
\caption{Summary of results. The quantities are $D_{\delta,K}=c_{\delta}\ln(K+1), \ c_{\delta,H}=c_{\delta}^{'}\mathrm{Cond}(H^TH)^{3},\ c_{H}=\mathrm{Cond}(H^TH)^{3},\ c_{\Theta}=\mathrm{Cond}(\Theta)^{3}$ and $\Theta=G^T\Sigma G$ with $G$ defined in \cref{Subsec:LinearModel}. $\delta>0$ and the bounds are obtained with probability higher than $1-\delta$.}\label{Tab:Summary}
\end{table}

\section{Main results}\label{Sec:ProjOperatorMain}
This section consists of two main parts, the first presents the upper bounds attained thanks to an assumption on the distribution parameters of the noise, while the second demonstrates how to circumvent the previous assumptions through regularization. In this section, we concentrate our attention on Scenarios 2 and 4, the results related to the others scenarios are displayed in the appendix for the ease of exposition. \cref{Sec:AppendixA} in Appendix provides more precise bounds than those stated in this section. They were simplified for the sake of comprehension.
The proofs are postponed to \cref{Sec:ProjOperatorSkeleton} for the upper bounds obtained thanks to an assumption and to \cref{Sec:ProjOperatorRegularizedProof} for the regularized estimators of projection operators.
\subsection{Projection operator estimation}\label{Subsec:ProjOperatorEstimation}
We first state the result for Scenario 2 where all the columns of $E$ are independent.
\begin{theorem}\label{Th:ProjOperatorRowsI}
Let $\delta\in(0,1)$. Consider Scenario 2 and assume that 
\begin{equation}\label{Eq:ProjOperatorRowsAssumI}
     \frac{1}{4}\rho_{\min}(H^TH)>d_{\delta}n\gamma^{2}\rho(S),
\end{equation}
with $d_{\delta}$ with depending only on $\delta$. Then with probability higher than $1-\delta$, then exists a constant~$D_{\delta}^{(1)}>0$, depending only on $\delta$, such that 
\begin{equation}\label{Eq:BoundProjOperatorRows}
\frac{1}{n}|||P_{[H]}-P_{[\hat{H}]}|||^{2}\le D_{\delta}^{(1)}\gamma^{2}\mathrm{Cond}(H^TH)^{3}\frac{\rho(S)}{\rho_{\min}(H^TH)}.
\end{equation}
\end{theorem}

A sketch of proof is presented in Appendix \ref{Sec:ProjOperatorSkeleton} which holds for all scenarios. A detailed version of $D^{(1)}_{\delta}$ is given in \cref{Subsub:Scen12} thanks to \cref{Th:ProjOperatorSkeletonProof} and respectively \cref{Sec:ProjOperatorAdapt} for a detailed writing of $d_{\delta}$.

\cref{Th:ProjOperatorRowsI} requires that the smallest eigenvalue of the Gram matrix $H^TH$ needs to be large enough to ensure the consistency of the estimation of the projection operator $P_{[H]}$. This level is related to the number of observations $n$ and the noise level $\gamma^{2}$. Assumption \eqref{Eq:ProjOperatorRowsAssumI} can be viewed as a distinction of the subspace relative to the noise, which would lead to a poor estimation of $H$ if the dimension $n$ or the noise level are too high. Technically Assumption \eqref{Eq:ProjOperatorRowsAssumI} is required to ensure the proper behavior of $(\hat{H}^T\hat{H})^{-1}.$ The obtained upper bound is dependent on the ratio $\gamma^{2}/\rho_{\min}(H^TH)$ where the comparison of the noise level with the smallest eigenvalue appears. As for estimation of projection operator on eigensubspaces in \textcite{ReissWahl} we recover the dependence in eigenvalues in the upper bound like in \cref{Eq:ReissWahlSpectralProj}. The term $\mathrm{Cond}(H^TH)$ in the final bound plays a fundamental role in the estimation of the projection operator. The comparison between the matrix $H$ with the smallest eigenvalue of the Gram matrix $H^TH$  and the noise level can be controlled thanks to $\rho(S)$. The upper obtained bound is related to the spectrum of the covariance matrix $S$. We can highlight the spectrum of the matrix $H^TH$ in the the right hand side of \cref{Eq:BoundProjOperatorRows} which describes how the approximation depends on the choice of the basis 
$H$ of the subspace $[H]$ and how the estimation of $H^TH$ can affect the estimation of the projection operator $P_{[H]}$.
The bound obtained here can be linked to the one related to Scenario 1 just by adding $\rho(S)$ (see discussion of \cref{Th:ProjOperatorIndependent} in Appendix \ref{Sec:AppendixA}). This highlights the fact that adding dependence in rows does not change significantly the behavior of the estimator. The same procedure can be done for the dependence in columns in Scenario 3 (see \cref{Th:ProjOperatorCollumns} in Appendix \ref{Sec:AppendixA}).

\begin{theorem}\label{Th:ProjOperatorGeneralized}
Let $\delta\in(0,1)$. Consider Scenario 4 and assume that 
\begin{equation}\label{Eq:ProjOperatorGeneralizedAssum}
     \frac{1}{4}\rho_{\min}(H^TH)>d_{\delta,K}\gamma^{2}\sum_{i=1}^{K}\mathrm{Tr}(A_{i}VA_{i}^{T}),
\end{equation}
with $d_{\delta,K}=b_{\delta}\ln\big(\frac{K}{\delta}\big)$ and $b_{\delta}$ depending only on $\delta$.
Then with probability higher than $1-\delta$, there exists a constant $D^{(2)}_{\delta}>0$, such that
\begin{equation*}
\frac{1}{n}|||P_{[H]}-P_{[\hat{H}]}|||^{2}\le \gamma^{2}D_{\delta}^{(2)}\ln\big(\frac{K}{\delta}\big)\mathrm{Cond}(H^TH)^{3} \bigg(\displaystyle\sum_{i=1}^{K}\frac{\mathrm{Tr}(A_{i}VA_{i}^{T})}{n\rho_{\min}(H^TH)}\bigg).
\end{equation*}
\end{theorem}

As for \cref{Th:ProjOperatorRows}, a sketch of proof is given in Appendix \ref{Sec:ProjOperatorSkeleton}. The detailed proof for this specific scenario is given in \cref{Subsub:Scen4}. See \cref{Subsub:Scen12} for an expression of $D^{(2)}$ and
\cref{SubsubsecScen4} for a writing of $d_{\delta}$.

This scenario is very specific and aims to be a generalization of the PLS case, where the subspace in question is a Krylov space. In the PLS context, the subspace $\mathcal{G}=[\sigma,\Sigma\sigma,...,\Sigma^{K-1}\sigma]$ is estimated through the use of $\hat{\sigma}$. Here in Scenario 4, the considered vector is $v$ and a basis of $H$ is $\{A_{1}v,...,A_{K}v\}$. Assumption \eqref{Eq:ProjOperatorGeneralizedAssum} is quite similar to the previous ones, it compares the smallest eigenvalue of $H^TH$ to the variance of each columns. What results is then the sum of the variances of each columns which is expected below the smallest eigenvalue to ensure a good estimation of the inverse of the Gram matrix. The term $\ln(K)$ is a counterpart related to the number of vectors (number of columns) on which we apply our concentration inequalities. The final upper bound is similar to the others bounds obtained for the previous scenarios with a ratio between the variance terms of the columns and the eigenvalues of the Gram matrix. The only difference is the counterpart $\ln(K)$.

All the previous theorems state upper bounds on the operator norm of the difference between the projection operator and its estimation up to an assumption on the Gram matrix $H^TH$. As seen in \textcite{Cast2} on PLS regression with $K$ components, the matrix can cause several issues if it is ill-conditioned. To circumvent these problematic situations, we introduce a Ridge regularization in each scenario.
\subsection{Regularized versions}\label{Subsec:ProjOperatorRegularized} In this section, regularized versions of the previous estimators are introduced by replacing $(\hat{H}^T\hat{H})^{-1}$ by $(\hat{H}^T\hat{H}+\alpha I_{K})^{-1}$. The regularized projection operator $P_{[\hat{H}]_{\alpha}}$ is defined as follows:
\begin{equation}\label{Eq:ProjOperatorRegularizedEstimator}
P_{[\hat{H}]_{\alpha}}=\hat{H}(\hat{H}^T\hat{H})_{\alpha}^{-1}\hat{H}^T,
\end{equation}
where $(\hat{H}^{T}\hat{H})_{\alpha}=(\hat{H}^T\hat{H}+\alpha I_{K})$.

The regularization introduced in \cref{Eq:ProjOperatorRegularizedEstimator} allows us to avoid making an assumption on the matrix $H$, that is, to get rid of Assumption \eqref{Eq:ProjOperatorRowsAssumI} and \eqref{Eq:ProjOperatorGeneralizedAssum}. By considering the regularized operator we assume the matrix $H$ has full rank $K$, which is quite natural. The parameter $\alpha$ is related to the variance term of the distribution on the matrix $E$ depending on the considered scenario. This approach succeeds for all the theorems stated previously, as detailed below (see Appendix for the others scenarios). 

\begin{theorem}\label{Th:ProjOperatorRowsPenaliniseI}
Let $\delta\in (0,1)$. Consider Scenario 2 and set 
\begin{equation}
\alpha=2d_{\delta}n\gamma^{2}\rho(S).
\end{equation}
Then with probability larger than $1-\delta$, there exists a constant $\tilde{D}^{(1)}_{\delta}>0$ such that
\begin{equation}
    \frac{1}{n}|||P_{[H]}-P_{[H]_{\alpha}}|||^{2}\le \tilde{D}^{(1)}_{\delta}\gamma^{2}\mathrm{Cond}(H^TH)^{3}\frac{\rho(S)}{\rho_{\min}(H^TH)}.
\end{equation}
\end{theorem}
\cref{Th:ProjOperatorRowsPenaliniseI} provides a similar upper bound to \cref{Th:ProjOperatorRowsI} in the context of Scenario 2. The parameter $\alpha$ is set depending of the variance term $S$. A general proof outline is given in Appendix \ref{Sec:ProjOperatorRegularizedProof} which holds for all scenarios. The details related to this scenario are given in \cref{Subsub:Scen12Reg}.

We recover the bound stated in \cref{Th:ProjOperatorRowsI}. The case of the dependence on the rows (Scenario 2) is then naturally transferred to the case of dependence on the columns with Scenario 3 (see the Appendix).
We also provide a regularized version of \cref{Th:ProjOperatorGeneralized} for Scenario 4. In a sense, regularization helps counterbalance the effect of noise, which may degrade the rank of the matrix $\hat{H}$.
\begin{theorem}\label{Th:ProjOperatorGeneralizedPenalinise}
Let $\delta\in (0,1)$. Consider Scenario 4 and set 
\begin{equation}
\alpha=2d_{\delta,K}\gamma^{2}\sum_{i=1}^{K}\mathrm{Tr}(A_{i}VA_{i}^{T}).
\end{equation}
Then with probability larger than $1-\delta$, there exists a constant $\tilde{D}^{(2)}_{\delta}>0$ such that
\begin{equation}
    \frac{1}{n}|||P_{[H]}-P_{[H]_{\alpha}}|||^{2}\le \tilde{D}^{(2)}_{\delta}\gamma^{2}\ln(\frac{K}{\delta})\mathrm{Cond}(H^TH)^{3}\displaystyle\sum_{i=1}^{K}\frac{\mathrm{Tr}(A_{i}VA_{i}^T)}{n\rho_{\min}(H^TH)}.
\end{equation}
\end{theorem}
The proof and details related to this scenario are presented in Appendix \ref{Sec:ProjOperatorRegularizedProof} and \cref{Subsub:Scen4Reg}.

The upper bound is similar to the one given in \cref{Th:ProjOperatorGeneralized}. The parameter $\alpha$ is naturally related to \cref{Eq:ProjOperatorGeneralizedAssum} and depends on $\ln(K/\delta)$. We recover a ratio between $\frac{\gamma^{2}}{n}\displaystyle\sum_{i=1}^{K}\mathrm{Tr}(A_{i}VA_{i}^{T})$, which can be related to the distribution of the noise through the matrix $E$, and the smallest eigenvalue $\rho_{\min}(H^TH)$ linked to the matrix $H$.
\subsection{The PLS case}\label{Subsec:ProjOperatorPLSCase}
As stated previously in \cref{Subsec:LinearModel}, PLS regression can be seen as a projection on a Krylov subspace.
Let consider the PLS regression over the subspace spanned by $H=XG\in\RR^{n\times K}$ where $G=(\sigma,\Sigma\sigma,...,\Sigma^{K-1}\sigma)\in\RR^{p\times K}$.
We assume that $H$ is full rank. We set $\hat{v}=X\hat{\sigma}$ where $v=X\sigma$. We are hence in Scenario 4. The matrix $V$ corresponds here to $\Sigma$, the noise level is $\gamma^{2}=\frac{\tau^{2}}{n}$, and for all $i\in\{1,...,K\}, A_{i}=X\Sigma^{i-1}$.
Assumption \eqref{Eq:ProjOperatorGeneralizedAssum} applied for the PLS regression becomes 
\begin{equation}\label{Eq:ProjOperatorPLSApp}
\frac{\rho_{\min}(G^TX^TXG)}{4}\ge d_{\delta,K}\frac{\tau^{2}}{n}\sum_{i=1}^{K}\mathrm{Tr}(X\Sigma^{2i-1}X^T).
\end{equation}
Let $\Theta=\frac{1}{n}G^TX^TXG$.

\cref{Eq:ProjOperatorPLSApp} can be rewritten in order to obtain
\begin{equation}\label{Eq:ProjOperatorPLSAss}
\frac{\rho_{\min}(\Theta)}{4}\ge d_{\delta,K}\frac{\tau^{2}}{n}\sum_{i=1}^{K}\mathrm{Tr}(\Sigma^{2i}).
\end{equation}
\begin{coro}
Assume \cref{Eq:ProjOperatorPLSAss} holds, \cref{Th:ProjOperatorGeneralized} establishes that with probability higher than $1-\delta,$
\begin{equation}\label{Eq:ProjOperatorPLSCase}
\frac{1}{n}|||P_{[XG]}-P_{[X\hat{G}]}|||^{2}\le D_{\delta}\ln(\frac{K}{\delta})\frac{\tau^{2}}{n}\mathrm{Cond}(\Theta)^{3}\frac{\sum_{i=1}^{K}\mathrm{Tr}(\Sigma^{2i})}{n\rho_{\min}(\Theta)}.
\end{equation}
\end{coro}
This result on the subspace spanned by the PLS components is obtained thanks to the fact that we assume that $\rho_{\min}(\Theta)$ is above a certain level depending on the eigenstructure of the Gram matrix$~\Sigma$.

\paragraph*{}Note that we estimate the difference between $[H]$ and $[\hat{H}]$ using the estimation of the corresponding bases, which is not the same as considering prediction for PLS regression. Indeed, PLS regression projects the response vector $Y$ onto the Krylov basis which is dependent with a vector of the basis $\hat{\sigma}=\frac{X^TY}{n}$. Here, we refer to the estimation of the subspace spanned by the components, rather than the quadratic prediction error related to the parameter $\beta$. These two notions can be related via the following decomposition \cref{Eq:ModelLinearDecomp}.
This decomposition highlights the difference between the given bounds on the estimation of the projection operator (\cref{Eq:ProjOperatorPLSCase}) and the study of the quadratic loss of the PLS estimator (Theorem 3.1 in \textcite{Cast2}).
\paragraph*{Connection to previous work}Assumption \eqref{Eq:ProjOperatorPLSAss} can be seen as a generalization of Assumption A.2 from \textcite{Cast2} on the norm of the Krylov components. In the later, the authors make an assumption on the diagonal of $\Theta$. The main difference here lies in the fact that Assumption \eqref{Eq:ProjOperatorPLSAss} is made directly on the minimal eigenvalue $\rho_{\min}(\Theta)$ of $\Theta.$ Assumption \eqref{Eq:ProjOperatorPLSAss} constraints the subspace spanned by the Krylov components to ensure a good estimation of the inverse of the matrix $\Theta$. The bound displayed in \cref{Eq:ProjOperatorPLSCase} is composed of several quantities. First $\mathrm{Cond}(\Theta)^{3}$ measures the invertibility of the Gram matrix $\Theta$ induced by the Krylov components. The trace of powers of the Gram matrix of $\Sigma$. Finally the ratio $\frac{\tau^{2}}{n^{2}}$ related to the noise $\frac{\tau^{2}}{n}$ of each estimation of the Krylov components with the normalization factor $\frac{1}{n}$. Several connections can be made with the bound displayed in Theorem 3.1 from \textcite{Cast2}. The trace of powers of the Gram matrix $\Sigma$ are common to both bounds. Focusing on the term $|||P_{[XG]}-P_{[X\hat{G}]}|||^{2}$, we recover the rate $\frac{\tau^{2}}{n}$ from the quadratic prediction loss of Theorem 3.1 from \textcite{Cast2}, reflecting how the estimation of the projection operator is related to the prediction of $X\beta$ by $X\hat{\beta}_{K}$.

\paragraph*{Regularized version} We also state a bound for the regularized estimator of the PLS subspace thanks to \cref{Th:ProjOperatorGeneralizedPenalinise}. Let 
\begin{equation}
    \alpha=d_{\delta,K}\frac{\tau^{2}}{n}\sum_{i=1}^{K}\mathrm{Tr}(\Sigma^{2i}).
\end{equation}
We have with probability higher than $1-\delta,$
\begin{equation}\label{Eq:ProjOperatorPLSRidge}
    \frac{1}{n}|||P_{[XG]}-P_{[X\hat{G}]_{\alpha}}|||^{2}\le D_{\delta}\ln(\frac{K}{\delta})\frac{\tau^{2}}{n}\mathrm{Cond}(\Theta)^{3}\frac{\sum_{i=1}^{K}\mathrm{Tr}(\Sigma^{2i})}{n\rho_{\min}(\Theta)}.
\end{equation}
Remark that the Ridge penalization introduced here is quite different from the one introduced in \textcite{Cast2}. The difference lies in the definition of $\alpha$ which is calibrated according to each diagonal element of the matrix $\hat{\Theta}$ in Theorem 3.2 from \textcite{Cast2}. In contrast, for the bound in \cref{Eq:ProjOperatorPLSRidge}, $\alpha$ is fixed for all the Krylov components. However we can associate these two regularizations by noticing that the term $\alpha$ from \cref{Eq:ProjOperatorPLSRidge} can be related to the sum of the calibrated $\alpha_{i}$ of Theorem 3.2 in \textcite{Cast2}. This can be interpreted as a general constraint on the entire set of Krylov components simultaneously.
\section{Conclusion}
In this contribution we focused on the approximation of a subspace through projection operators estimation. We provide non-asymptotic upper bounds on the estimation error for four different scenarios. The latter are obtained thanks to a signal-to-noise condition on the Gram matrix related to the estimated subspace. Under this assumption, the displayed bounds depend on a ratio between the spectrum of the covariance matrix of the noise and the smallest eigenvalue of the Gram matrix. Each scenario can be linked to a specific case, in particular scenario 4 is applied for dimension reduction for PLS regression in a linear context.  

In order to get rid of the signal-to-noise condition on the Gram matrix, we introduce a Ridge regularization in the estimator of the projection operator. This method avoids making an assumption on the Gram matrix thanks to a parameter $\alpha$ depending only of the dimension and the covariance matrix of the noise. The obtained upper bounds for the regularized estimator are similar to those previously mentioned under the signal-to-noise condition. It would be interesting to assess the optimality of the obtained upper bounds by emphasizing on lower bounds. Considering the case of dimension 1 would be a first step to investigate.

\appendix
\section*{Appendix}
The appendix is structured as follows, Appendix \ref{Sec:AppendixA} gives all the precise upper bounds according to each scenarios. Appendix \ref{Sec:ProjOperatorPrelim} is dedicated to some specific technical results that will be used all along the proofs. Appendix \ref{Sec:ProjOperatorSkeleton} states the general sketch of proof which will be applied for each scenario. Next, we highlight and fix the quantities according to each scenario to support the proof. Then we bound the introduced terms according to each scenario to conclude the proof of the theorems (without regularization). Appendix \ref{Sec:ProjOperatorRegularizedProof} gives the general proof for the regularized versions which adapts according to the terms of each scenarios. Finally, we provide all the upper bounds according to the introduced terms of each scenarios with regularization.
\section{Precise results}\label{Sec:AppendixA}
The bound displayed for Scenario 2 in \cref{Sec:ProjOperatorMain} has been simplified for the ease of exposition. The detailed results are displayed below. We first provide the theorems related to each Scenario, then we give the bounds obtained thanks to a regularization on the Gram matrix $H^TH$. 
\subsection{Detailed results in each Scenario}
First we consider the Scenario where all entries of the matrix E are $iid$. As mentioned previously in \cref{Sec:ProjOperatorMain}, the obtained result is simply the one from \cref{Th:ProjOperatorRowsI} with $S=I_{p}$.
\begin{theorem}\label{Th:ProjOperatorIndependent}
Let $\delta\in(0,1)$. Suppose the assumptions of Scenario 1 are satisfied and assume that 
\begin{equation}\label{Eq:ProjOperatorIndependentAssum}
     \frac{1}{4}\rho_{\min}(H^TH)>d_{\delta}n\gamma^{2},
\end{equation}
with $d_{\delta}$ depending only on $\delta$. Then, with probability higher than $1-\delta$,
\begin{equation}
\frac{1}{n}|||P_{[H]}-P_{[\hat{H}]}|||^{2}\le \gamma^{2}D_{\delta}\mathrm{Cond}(H^TH)^{3} \bigg(\frac{1}{\rho_{\min}(H^TH)}\bigg).
\end{equation}
\end{theorem}

In section \cref{Sec:ProjOperatorMain}, \cref{Th:ProjOperatorRowsI} has been simplified for the sake of clarity, a more precise bound is displayed below. 
\begin{theorem}\label{Th:ProjOperatorRows}
Let $\delta\in(0,1)$. Suppose the assumptions of Scenario 2 are satisfied and assume that 
\begin{equation}\label{Eq:ProjOperatorRowsAssum}
     \frac{1}{4}\rho_{\min}(H^TH)>d_{\delta}n\gamma^{2}\rho(S),
\end{equation}
with $d_{\delta,K}$ with depending only on $\delta$. Then with probability higher than $1-\delta$,
\begin{multline*}
\frac{1}{n}|||P_{[H]}-P_{[\hat{H}]}|||^{2}\le D_{\delta}\gamma^{2}\mathrm{Cond}(H^TH)\\ \max\bigg(\frac{K}{n}\mathrm{Cond}(H^TH)^{2}\rho\big(S(H^TH)^{-1}\big), \rho\big(S(H^TH)^{-1}\big),\mathrm{Cond}(H^TH)^{2}\frac{\mathrm{Tr}(S)}{n\rho_{\min}(H^TH)}\bigg).
\end{multline*}
\end{theorem}

In Scenario 3, considering a distribution on the columns of the matrix $E$ is quite similar to Scenario~2, we provide a detailed upper bound and a simplified one for the ease of exposition.
\begin{theorem}\label{Th:ProjOperatorCollumns}
Let $\delta\in(0,1)$. Suppose the assumptions of Scenario 3 are satisfied and assume that 
\begin{equation}\label{Eq:ProjOperatorColumnsAssum}
     \frac{1}{4}\rho_{\min}(H^TH)>d_{\delta}n\gamma^{2}\rho(A),
\end{equation}
with $d_{\delta}$ depending only on $\delta$. Then with probability higher than $1-\delta$,
\begin{multline*}
\frac{1}{n}|||P_{[H]}-P_{[\hat{H}]}|||^{2}\le \gamma^{2}D_{\delta}\mathrm{Cond}(H^TH)^{2}\\ \max\bigg(\frac{\max\big(K\rho(A),\mathrm{Tr}(A)\big)}{n\rho_{\min}(H^TH)},\mathrm{Cond}(H^TH)\frac{\rho(A)}{n}\mathrm{Tr}\big((H^TH)^{-1}\big)\bigg).
\end{multline*}
\end{theorem}
Assumption \eqref{Eq:ProjOperatorColumnsAssum} for Scenario 3 is the same as \eqref{Eq:ProjOperatorRowsAssum} in Scenario 2. The displayed upper bound is also related to the spectrum of the covariance matrix $A$. As for \cref{Th:ProjOperatorRows} the bound can be simplified.  
\begin{coro}\label{Cor:ProjOperatorColumns}
Under the same assumption,
\begin{equation*}
    \frac{1}{n}|||P_{[\hat{H}]}-P_{[H]}|||^{2}\le \gamma^{2}D_{\delta}\mathrm{Cond}(H^TH)^{3}\max\bigg(\frac{K}{n}\frac{\rho(A)}{\rho_{\min}(H^TH)},\frac{\mathrm{Tr}(A)}{n\rho_{\min}(H^TH)}\bigg).
\end{equation*}
\end{coro}
This bound represents a ratio between the noise relative to the eigenvalues of $A$ and the "signal" in $H$ expressed by $\rho_{\min}(H^TH)$.

\subsection{Detailed results with Ridge regularization}
The assumption on the spectrum of the Gram matrix $H^TH$ can be removed in each Scenario thanks to a Ridge penalization. The regularization parameter is directly related to each assumption in order to ensure a good estimate of the projection operator. 
\begin{theorem}\label{Th:ProjOperatorIndependentPenalinise}
Let $\delta\in (0,1)$. Suppose the assumptions of Scenario 1 are satisfied and set 
\begin{equation}
\alpha=2d_{\delta}\gamma^{2}n.
\end{equation}
Then with probability larger than $1-\delta$,
\begin{equation}
    \frac{1}{n}|||P_{[H]}-P_{[H]_{\alpha}}|||^{2}\le \gamma^{2}D_{\delta}\left(\mathrm{Cond}(H^TH)^{3}\frac{1}{\rho_{\min}(H^TH)}\right).
\end{equation}
\end{theorem}
\cref{Th:ProjOperatorIndependentPenalinise} can easily be generalized to Scenario 2 by adding the dependence on each row of the matrix $E$.
\begin{theorem}\label{Th:ProjOperatorRowsPenalinise}
Let $\delta\in (0,1)$. Suppose the assumptions of Scenario 2 are satisfied and set 
\begin{equation}
\alpha=2d_{\delta}n\gamma^{2}\rho(S).
\end{equation}
Then with probability larger than $1-\delta$,
\begin{equation}
    \frac{1}{n}|||P_{[H]}-P_{[H]_{\alpha}}|||^{2}\le\gamma^{2} D_{\delta}\max\left(\mathrm{Cond}(H^TH)^{3}\frac{\rho(S)}{\rho_{\min}(H^TH)},\frac{\mathrm{Tr}(S)}{n\rho_{\min}(H^TH)}\right).
\end{equation}
\end{theorem}
\cref{Th:ProjOperatorRowsPenalinise} provides a similar upper bound to \cref{Th:ProjOperatorRows} in the context of Scenario 2. The parameter $\alpha$ is set depending of the variance term $S$. The upper bound can also be simplified for the sake of clarity.
\begin{coro}
Let $\delta\in(0,1)$. Suppose $H$ full rank and define $$\alpha=2d_{\delta}n\gamma^{2}\rho(S).$$ Then, we have with probability higher than $1-\delta$,
\begin{equation}
    \frac{1}{n}|||P_{[H]}-P_{[H]_{\alpha}}|||^{2}\le D_{\delta}\gamma^{2}\mathrm{Cond}(H^TH)^{3}\frac{\rho(S)}{\rho_{\min}(H^TH)}.
\end{equation}
\end{coro}
We recover the bound stated in \cref{Th:ProjOperatorRowsI}. The case of the dependence on the rows (Scenario 2) is then naturally transferred to the case of dependence on the columns with Scenario 3.

\begin{theorem}\label{Th:ProjOperatorColumnsPenalinise}
Let $\delta\in (0,1)$. Suppose the assumptions of Scenario 3 are satisfied and set 
\begin{equation}
\alpha=2d_{\delta}n\gamma^{2}\rho(A).
\end{equation}
Then with probability larger than $1-\delta$,
\begin{multline*}
    \frac{1}{n}|||P_{[H]}-P_{[H]_{\alpha}}|||^{2}\le\\ D_{\delta}\max\left(\mathrm{Cond}(H^TH)^{3}\frac{K\rho(A)}{n\rho_{\min}(H^TH)},\mathrm{Cond}(H^TH)^{3}\frac{\mathrm{Tr}(A)}{n\rho_{\min}(H^TH)},\frac{\rho(A)}{\rho_{\min}(H^TH)}\right).
\end{multline*}
\end{theorem}
In Scenario 3, we obtain a similar result to Scenario 2 \cref{Th:ProjOperatorCollumns}. A natural corollary is provided as for \cref{Cor:ProjOperatorColumns}.
\begin{coro}
Under Scenario 3, and assumptions of \cref{Th:ProjOperatorRows},
\begin{equation*}
    \frac{1}{n}|||P_{[\hat{H}]_{\alpha}}-P_{[H]}|||^{2}\le \gamma^{2}D_{\delta}\mathrm{Cond}(H^TH)^{3}\max\bigg(\frac{K}{n}\frac{\rho(A)}{\rho_{\min}(H^TH)},\frac{\mathrm{Tr}(A)}{n\rho_{\min}(H^TH)}\bigg).
\end{equation*}
\end{coro}

\section{Preliminaries}\label{Sec:ProjOperatorPrelim}
We state deviation results based on \textcite{Vershynin2012} on the eigenvalues of random matrices. The inequalities are non-asymptotic. We will present the different scenarios discussed and apply the results on each of these. The first result is naturally based on centered Gaussian variables entries. 
\begin{proposition}(Vershynin 2012)\label{Prop:SingularValueVershynin}
Let $X\in\RR^{n\times K}$ where $X_{ij}$ follows a standard normal distribution.
For every $t\ge 0$, we have with probability higher than $1-2\mathrm{exp}(-\frac{t^{2}}{2})$
$$\sqrt{n}-\sqrt{K}-t\le s_{\min}(X)\le s_{\max}(X)\le\sqrt{n}+\sqrt{K}+t,$$
where $s_{\min}$ and $s_{\max}$ are respectively the minimal and the maximal singular values of the matrix $X$.
\end{proposition}
A relevant way to state this result is to consider the Gram matrix associated with $X$, which links the singular valuesto the eigenvalues of $\frac{1}{n}X^TX$. The statement is the following.
\begin{coro}(Vershynin 2012)\label{Cor:VershyninGram}
Under the assumptions of \cref{Prop:SingularValueVershynin}, for $t>0$ we have with probability higher than $1-\mathrm{exp}(-\frac{t^{2}}{2})$,
$$|||\frac{1}{n}X^TX-I_{K}|||\le 2\varepsilon_{K,n,t}+\varepsilon_{K,n,t}^{2},$$
where $\varepsilon_{K,n,t}=\sqrt{\frac{K}{n}}+\frac{t}{\sqrt{n}}$.
\end{coro}
This corollary shows that the estimation of the covariance matrix $I_{K}$ by the sample covariance matrix directly depends of the ratio $\max\left(\frac{K}{n},\sqrt{\frac{K}{n}}\right)$. This standard case can be naturally generalized to the setting where the rows $X_{i}^{T}\sim\mathcal{N}(0,S)$ with $S\in\mathbb{R}^{K\times K}$.  
\begin{proposition}\label{Prop:VershyninRows}
Let $X\in\mathbb{R}^{n\times K}$ where $X_{i}^{T}\sim\mathcal{N}(0,S)$ are $iid$ for $i\in\{1,...,n\}$. Let $\delta\in(0,1)$, we have with probability higher than $1-\delta$,
\begin{equation}
|||\frac{1}{n}X^TX-S|||\le 4|||S|||D_{\delta}\max\left(\sqrt{\frac{K}{n}},\frac{K}{n}\right),
\end{equation}
with $D_{\delta}=1+2\ln(\frac{2}{\delta})+2\sqrt{\ln(\frac{2}{\delta})}.$
\end{proposition}
\begin{proof}
We use the fact that $X=ZS^{\frac{1}{2}}$ where $Z$ entries are standard Gaussian variables. 
Thanks to \cref{Cor:VershyninGram}, we have with probability higher than $1-2\mathrm{exp}(-\frac{t^{2}}{2})$, 
\begin{align*}
|||\frac{1}{n}X^TX-S|||&=|||S^{\frac{1}{2}}Z^TZS^{\frac{1}{2}}-S|||\\
&\le |||\frac{1}{n}Z^TZ-I_{K}|||\cdot|||S|||\\
&\le |||S|||\cdot(2\varepsilon_{K,n,t}+\varepsilon_{K,n,t}^{2}).
\end{align*}
By setting $t=\sqrt{2\ln(\frac{2}{\delta})}$ we get the desired result.
\end{proof}
We refer to \textcite{Barzilai} for more details on the estimation of the $i$-th eigenvalue of a Gram matrix under different scenarios.

\paragraph*{}As in \textcite{Cast} and \textcite{Cast2}, we will use a deviation result for some quadratic function of Gaussian vectors from \textcite{Laurent} in order to obtain non-asymptotic upper bounds.
\begin{proposition}(\textbf{Deviation result})\label{Prop:DeviationResult}
Let $u\sim\mathcal{N}(0,tU)$ with $t\in\mathbb{R}_{+}$ and $U\in\mathbb{R}^{D\times D}$ a symmetric positive matrix. 
For $x\ge0$, we have, with probability higher than $1-e^{-x},$
$$\|u\|^{2}\le g(x)t\mathrm{Tr}(U^{2}),$$
with $g(x)=1+2x+2\sqrt{x}$. Setting $x_{\delta}=\ln(1/\delta)$, we have, with probability higher than $1-\delta,$
$$\|u\|^{2}\le C_{\delta}t\mathrm{Tr}(U^{2}),$$
with $C_{\delta}=g(x_{\delta})$.
\end{proposition}

\section{Proof for the estimation of projection operator}\label{Sec:ProjOperatorSkeleton}
\subsection{Sketch of proof}
This section is dedicated to the proofs of \cref{Th:ProjOperatorIndependent}, \cref{Th:ProjOperatorRows}, \cref{Th:ProjOperatorCollumns} and \cref{Th:ProjOperatorGeneralized}. We propose a general proof framework that adapts to each scenario. We consider here the main procedure to prove the different results cited above.
Let $u\in\mathbb{R}^{n}$ such that $u^Tu=1$. We consider 
$$\frac{1}{n}\|P_{[H]}(u)-P_{[\hat{H}]}(u)\|^{2}.$$

The proof involves three steps, 
\begin{enumerate}
\item First considering an assumption which leads to the inversion of $\hat{H}^T\hat{H}$ and different consequences of this assumption. 
\item As a second step we will consider the bounds of three main terms denoted by $\mathrm{I},\mathrm{II}$ and $\mathrm{III}$. 
\item Finally taking into account the supremum on all $u$ such that $\|u\|=1$ to get the final bound on the estimation of the operator norm.
\end{enumerate}
Before stating the theorem, we introduce an event, a general assumption and three terms involved in the bound. 
\begin{definition}\label{Def:EventA}
Let $\delta\in(0,1)$. We define an event $\mathcal{A}_{\delta}$ satisfying $\mathbb{P}(\mathcal{A}_{\delta})\ge 1-\delta$, 
on which we have $\rho(E^TE)\le d_{\delta,K}\Psi$ where $d_{\delta,K}$ depending on $\delta$ and $K$ with $\Psi\ge0$. 
\end{definition}
According to the distribution on $E\in \RR^{n\times K}$, the terms $d_{\delta,K}$ and $\Psi$ will be precised later for each scenario. $\Psi$ is deterministic and depends of the parameters of the distribution of $E$ and is directly related to the variance term of the coefficients of $\hat{H}$. We specify that the dependence on $K$ for $d_{\delta,K}$ only occurs for Scenario 4. 
\begin{description}
\item[Assumption E.1.]\label{Ass:AssumptionE.1}\textit{The matrix $H^TH$ satisfies $\frac{1}{4}\rho_{\min}(H^TH)>d_{\delta,K}\Psi$, with $d_{\delta,K}$ depending only on $\delta$ and $K$}.
\end{description}

\paragraph*{}The three main terms are the following for $u\in\RR^{n}$ with $u^Tu=1$,
\begin{align}
    \mathrm{I}&=\|E(H^TH)^{-1}H^Tu\|^{2}\label{Eq:ProjOperatorI}\\
    \mathrm{II}&=\|\left(\hat{H}^T\hat{H}-H^TH\right)(H^TH)^{-1}H^Tu\|^{2}\label{Eq:ProjOperatorII}\\
    \mathrm{III}&=\|H(H^TH)^{-1}E^Tu\|^{2}.\label{Eq:ProjOperatorIII}
\end{align}
The dependence on $u$ is omitted for the sake of simplicity. 
For $u\in\mathbb{R}^{n}$ with $u^Tu=1$, we introduce $\Lambda_{u}\in\mathbb{R}^{K}$ defined as 
\begin{equation}\label{Eq:ProjOperatorLambda}
\Lambda_{u}=(H^TH)^{-1}H^Tu.
\end{equation}
\begin{theorem}\label{Th:ProjOperatorSkeletonProof}
Let $\delta\in(0,1)$. Under \textbf{Assumption E.1} we have, for $u\in\RR^{K}$ with $u^Tu=1$, on the event $\mathcal{A}_{\delta}$,
\begin{equation}
    \|P_{[H]}(u)-P_{[\hat{H}]}(u)\|^{2}\le \tilde{C}_{1}\mathrm{I}+\tilde{C}_{2}\mathrm{Cond}(H^TH)\rho_{\min}(H^TH)^{-1}\mathrm{II}+\tilde{C}_{3}\mathrm{Cond}(H^TH)^{3}\mathrm{III}.
\end{equation}
\end{theorem}
Before proving the theorem, we list results that directly follow from \textbf{Assumption E.1}. In order to ensure a control on the error of estimation $H^TH$ and its inverse, we consider in each case an assumption on the quantity $\rho_{\min}(H^TH)$.
\begin{proposition}\label{Prop:ConsequenceAssInverse}
Under \textbf{Assumption E.1} we have, on $\mathcal{A}_{\delta}$, 
$$\rho_{\min}(\hat{H}^T\hat{H})\ge \frac{\rho_{\min}(H^TH)}{2}.$$
\end{proposition}
\begin{proof}
Let $x\in\RR^{K}$ such that $x^Tx=1$.
\begin{align*}
    x^T\hat{H}^T\hat{H}x&=x^TH^THx+x^T(\hat{H}-H)^T(\hat{H}-H)x+2x^T(\hat{H}-H)^{T}Hx\\
    &\ge x^TH^THx-2|x^T(\hat{H}-H)^THx|\\
    &\ge \frac{3}{4}x^TH^THx-4x^T(\hat{H}-H)^T(\hat{H}-H)x\\
    &\ge \frac{3}{4}\rho_{\min}(H^TH)-4\rho\big((\hat{H}-H)^T(\hat{H}-H)\big).
\end{align*}
We use the fact that $\hat{H}-H=E$. Hence on $\mathcal{A}_{\delta}$ we have for all $x\in\RR^{K}$ with $x^Tx=1$,
$$x^T\hat{H}^T\hat{H}x\ge \frac{3}{4}\rho_{\min}(H^TH)-d_{\delta,K}\Psi, \text{by definition of } \mathcal{A}_{\delta}.$$
We get the desired result using \textbf{Assumption E.1}.
\end{proof}
\begin{proposition}\label{Prop:ConsequencesAssumption}
Under \textbf{Assumption E.1}, we get that, on $\mathcal{A}_{\delta}$,
\begin{align*}
    \rho(E^TE)&\le\frac{\rho_{\min}(H^TH)}{4},\\
    \rho(\hat{H}^T\hat{H}-H^TH)&\le 2\rho(H^TH).
\end{align*}
\end{proposition}
\begin{proof}
On $\mathcal{A}_{\delta}$, we have,
\begin{align*}
\rho(E^TE)&\le d_{\delta,K}\Psi\\
&\le \frac{\rho_{\min}(H^TH)}{4},
\end{align*}
where the last step results from \textbf{Assumption E.1}. Then for $x\in\RR^{p}$ such that $x^Tx=1$ we have 
\begin{align*}
|x^T(\hat{H}^T\hat{H}-H^TH)x|&\le |x^T(H^TE+E^TH+E^TE)x|\\
&\le 2x^TE^TEx+x^TH^THx\\
&\le 2\frac{\rho_{\min}(H^TH)}{4}+\rho(H^TH)\\
&\le 2\rho(H^TH).
\end{align*}
\end{proof}

Now we give the proof of main theorem by regrouping the three different terms thanks to the consequences of \textbf{Assumption E.1}
\begin{proof}[Proof of \cref{Th:ProjOperatorSkeletonProof}]
In the following we shall bound these terms using $\mathrm{I}, \mathrm{II}$ and $\mathrm{III}$. First,
$$P_{[H]}(u)-P_{[\hat{H}]}(u)=H(H^TH)^{-1}H^Tu-\hat{H}(\hat{H}^T\hat{H})^{-1}\hat{H}^{T}u.$$
Hence we have 
$$P_{[H]}(u)-P_{[\hat{H}]}(u)=\big(H-\hat{H}\big)(\hat{H}^T\hat{H})^{-1}\hat{H}^Tu+H\big((H^TH)^{-1}H^Tu-(\hat{H}^T\hat{H})^{-1}\hat{H}^Tu\big),$$
Which leads to 
\begin{multline*}
P_{[H]}(u)-P_{[\hat{H}]}(u)=\\
\left(H-\hat{H}\right)(\hat{H}^T\hat{H})^{-1}\hat{H}^Tu+H\left((H^TH)^{-1}-(\hat{H}^T\hat{H})^{-1}\right)\hat{H}^Tu+H(H^TH)^{-1}\left(H^Tu-\hat{H}^Tu\right).
\end{multline*}
Then we compute
\begin{align}
    \|P_{H}(u)-P_{\hat{H}}(u)\|^2&\le4u^T\hat{H}(\hat{H}^T\hat{H})^{-1}(\hat{H}-H)^T(\hat{H}-H)(\hat{H}^T\hat{H})^{-1}\hat{H}^Tu\nonumber\\
    &\quad+4u^T\hat{H}\big((\hat{H}^T\hat{H})^{-1}-(H^TH)^{-1}\big)H^TH\big((\hat{H}^T\hat{H})^{-1}-(H^TH)^{-1}\big)\hat{H}^Tu\nonumber\\
    &\quad+2(H^Tu-\hat{H}^Tu)^T(H^TH)^{-1}H^TH(H^TH)^{-1}(H^Tu-\hat{H}^Tu)\nonumber\\
    &:=4T_{1}+4T_{2}+2\mathrm{III}.\label{Eq:ProjOperatorDecoupage}
\end{align}
We focus on $T_{1}$,
\paragraph{Term $T_1$=}
$u^T\hat{H}(\hat{H}^T\hat{H})^{-1}(\hat{H}-H)^T(\hat{H}-H)(\hat{H}^T\hat{H})^{-1}\hat{H}^Tu$.
\begin{align}
T_1&\le2u^T\hat{H}(H^TH)^{-1}(\hat{H}-H)^T(\hat{H}-H)(H^TH)^{-1}\hat{H}^Tu\nonumber\\
&\quad+2u^T\hat{H}((\hat{H}^T\hat{H})^{-1}-(H^TH)^{-1})(\hat{H}-H)^T(\hat{H}-H)((\hat{H}^T\hat{H})^{-1}-(H^TH)^{-1})\hat{H}^Tu\nonumber\\
&=:T_{11}+T_{12}.\label{Eq:ProjOperatora}
\end{align}
Let us bound $T_{12}$.
\begin{align*}
    \lefteqn{T_{12}\le}\\
    &2u^T\hat{H}(H^TH)^{-1}(\hat{H}^T\hat{H}-H^TH)(\hat{H}^T\hat{H})^{-1}E^TE(\hat{H}^T\hat{H})^{-1}(\hat{H}^T\hat{H}-H^TH)(H^TH)^{-1}\hat{H}^Tu\\
    &\le2\frac{\rho\big(E^TE\big)}{\rho_{\min}(\hat{H}^T\hat{H})}u^T\hat{H}(H^TH)^{-1}(\hat{H}^T\hat{H}-H^TH)(\hat{H}^T\hat{H})^{-1}(\hat{H}^T\hat{H}-H^TH)(H^TH)^{-1}\hat{H}^Tu\\
    &\le 4\frac{\rho\left(E^TE\right)}{\rho_{\min}(\hat{H}^T\hat{H})}u^TE(H^TH)^{-1}(\hat{H}^T\hat{H}-H^TH)(\hat{H}^T\hat{H})^{-1}(\hat{H}^T\hat{H}-H^TH)(H^TH)^{-1}E^Tu\\
    &\quad+4\frac{\rho\left(E^TE\right)}{\rho_{\min}(\hat{H}^T\hat{H})}u^TH(H^TH)^{-1}(\hat{H}^T\hat{H}-H^TH)(\hat{H}^T\hat{H})^{-1}(\hat{H}^T\hat{H}-H^TH)(H^TH)^{-1}H^Tu.
\end{align*}
Then we use the fact that $\rho\big(E^TE\big) \le \frac{\rho_{\min}(H^TH)}{4}$
and $\rho_{\min}(\hat{H}^T\hat{H})\ge \frac{\rho_{\min}(H^TH)}{2}$ on $\mathcal{A}_{\delta}$ thanks to \cref{Prop:ConsequenceAssInverse} and \cref{Prop:ConsequencesAssumption}.
We get 
$$\frac{\rho\big((\hat{H}-H)^T(\hat{H}-H)\big)}{\rho_{\min}(\hat{H}^T\hat{H})}\le \frac{\rho_{\min}(H^TH)}{4}\frac{2}{\rho_{\min}(H^TH)}\le \frac{1}{2}.$$
We use the inequality $\rho(\hat{H}^T\hat{H}-H^TH)\le 2\rho(H^TH)$ from \cref{Prop:ConsequencesAssumption} to obtain the following, 
\begin{align}
    \lefteqn{T_{12}=}\nonumber\\
    &\ 2(u^T\hat{H}-u^TH)(H^TH)^{-1}(\hat{H}^T\hat{H}-H^TH)(\hat{H}^T\hat{H})^{-1}(\hat{H}^T\hat{H}-H^TH)(H^TH)^{-1}(\hat{H}^Tu-H^Tu)\nonumber\\
    &\quad+2u^TH(H^TH)^{-1}(\hat{H}^T\hat{H}-H^TH)(\hat{H}^T\hat{H})^{-1}(\hat{H}^T\hat{H}-H^TH)(H^TH)^{-1}H^Tu\nonumber\\
    &\le2 \frac{\rho\big(\hat{H}^T\hat{H}-H^TH\big)^2}{\rho_{\min}(\hat{H}^T\hat{H})\rho_{\min}(H^TH)}(u^T\hat{H}-u^TH)(H^TH)^{-1}H^TH(H^TH)^{-1}(\hat{H}^Tu-H^Tu)\nonumber\\
    &\quad+2u^TH(H^TH)^{-1}(\hat{H}^T\hat{H}-H^TH)(\hat{H}^T\hat{H})^{-1}(\hat{H}^T\hat{H}-H^TH)(H^TH)^{-1}H^Tu\nonumber\\
    &\le 16 \frac{\rho(H^TH)^2}{\rho_{\min}(H^TH)^2}\mathrm{III}+2\frac{2}{\rho_{\min}(H^TH)}\mathrm{II},\label{Eq:ProjOperatorb}
\end{align}
where we recall that $\mathrm{III}=(u^TH-u^T\hat{H})(H^TH)^{-1}H^TH(H^TH)^{-1}(H^Tu-\hat{H}^Tu)$\\ and $\mathrm{II}=u^TH(H^TH)^{-1}(\hat{H}^T\hat{H}-H^TH)^{2}(H^TH)^{-1}H^Tu$.
Next,
\begin{align}
    T_{11}&=
    \ 2u^T\hat{H}(H^TH)^{-1}(\hat{H}-H)^T(\hat{H}-H)(H^TH)^{-1}\hat{H}^Tu\nonumber\\
    &\le 4u^TH(H^TH)^{-1}(\hat{H}-H)^{T}(\hat{H}-H)(H^TH)^{-1}H^Tu\nonumber\\
    &\quad+4(u^T\hat{H}-u^TH)(H^TH)^{-1}(\hat{H}-H)^T(\hat{H}-H)(H^TH)^{-1}(\hat{H}^Tu-H^Tu)\nonumber\\
    &\le 4\mathrm{I}+4\frac{\rho\big((\hat{H}-H)^T(\hat{H}-H)\big)}{\rho_{\min}(H^TH)}\mathrm{III}\nonumber\\
    &\le 4\mathrm{I} + \mathrm{III}\label{Eq:ProjOperatorc}
\end{align}
\paragraph{Final bound on $T_{1}$}
We deduce from \eqref{Eq:ProjOperatora}, \eqref{Eq:ProjOperatorb} and \eqref{Eq:ProjOperatorc} that
\begin{equation}\label{Eq:ProjOperatorBorneT1}
    T_{1}\le 4\mathrm{I} + \big(1+16\mathrm{Cond}(H^TH)\big)\mathrm{III}+ 4\rho_{\min}(H^TH)^{-1}\mathrm{II}.
\end{equation}
\paragraph{Term $T_2$}
We now focus on $T_{2}=u^T\hat{H}\big((\hat{H}^T\hat{H})^{-1}-(H^TH)^{-1}\big)H^TH\big((\hat{H}^T\hat{H})^{-1}-(H^TH)^{-1}\big)\hat{H}^Tu$.
\begin{align*}
    \lefteqn{T_{2}=}\\
    &\ 2u^TH\big((\hat{H}^T\hat{H})^{-1}-(H^TH)^{-1}\big)H^TH\big((\hat{H}^T\hat{H})^{-1}-(H^TH)^{-1}\big)H^Tu\\
    &  \quad+ 2(u^T\hat{H}-u^TH)\big((\hat{H}^T\hat{H})^{-1}-(H^TH)^{-1}\big)H^TH\big((\hat{H}^T\hat{H})^{-1}-(H^TH)^{-1}\big)(\hat{H}^Tu-H^Tu)\\
    &\le 2u^TH(H^TH)^{-1}\big(\hat{H}^T\hat{H}-H^TH\big)(\hat{H}^T\hat{H})^{-1}H^TH(\hat{H}^T\hat{H})^{-1}\big(\hat{H}^T\hat{H}-H^TH\big)(H^TH)^{-1}H^Tu\\
    & \quad+ 2 u^TE(H^TH)^{-1}(\hat{H}^T\hat{H}-H^TH)(\hat{H}^T\hat{H})^{-1}H^TH(\hat{H}^T\hat{H})^{-1}(\hat{H}^T\hat{H}-H^TH)(H^TH)^{-1}E^Tu\\
    &\le 2\frac{\rho(H^TH)}{\rho_{\min}(\hat{H}^T\hat{H})^2}u^TH(H^TH)^{-1}(\hat{H}^T\hat{H}-H^TH)^2(H^TH)^{-1}H^Tu\\
    & \quad+ 2 \frac{\rho(H^TH)}{\rho_{\min}(\hat{H}^T\hat{H})^{2}}\frac{\rho\big(\hat{H}^T\hat{H}-H^TH\big)^2}{\rho_{\min}(H^TH)}\mathrm{III}\\
    &\le 8\frac{\rho(H^TH)}{\rho_{\min}(H^TH)^{2}}\mathrm{II}+32\frac{\rho(H^TH)^{3}}{\rho_{\min}(H^TH)^{3}}\mathrm{III}.
\end{align*}
\paragraph{Final bound on $T_{2}$}
We deduce that
\begin{equation}\label{Eq:ProjOperatorBorneT2}
    T_{2}\le8\mathrm{Cond}(H^TH)\rho_{\min}(H^TH)^{-1}\mathrm{II}+32\mathrm{Cond}(H^TH)^{3}\mathrm{III}.
\end{equation}
\paragraph{Final Bound.}
Using \eqref{Eq:ProjOperatorBorneT2}, \cref{Eq:ProjOperatorBorneT1} and \eqref{Eq:ProjOperatorDecoupage}, we get the following final bound
\begin{multline}
\|(P_{[\hat{H}]}-P_{[H]})(u)\|^{2}\le\\ 
16\mathrm{I}+ (16+32\mathrm{Cond}(H^TH))\rho_{\min}(H^TH)^{-1}\mathrm{II}+(5+4\cdot16\mathrm{Cond}(H^TH)+4\cdot32\mathrm{Cond}(H^TH)^{3})\mathrm{III}.
\end{multline}
This concludes the proof.
\end{proof}

\subsection{Preliminaries for each scenario}\label{Sec:ProjOperatorAdapt}
\cref{Th:ProjOperatorSkeletonProof} provides an upper bound on $\frac{1}{n}\|P_{[H]}(u)-P_{[\hat{H}]}(u)\|^{2}$ depending of three main terms up to \textbf{Assumption E.1}. Thanks to results in Appendix \ref{Sec:ProjOperatorPrelim} we will state the assumptions for each scenario and set the parameter $\Psi$ related to \textbf{Assumption E.1}  as $\Psi_{S_{i}}$ for the $i$-th Scenario introduced in Appendix \ref{Sec:ProjOperatorSkeleton}. 

\subsubsection{Scenario 1 and 2}
Scenario 1 can be seen as a direct consequence of Scenario 2 with $S=I_{K}$.
In order to define the main assumption for Scenario 2, we have to define the function $\Psi_{S_{2}}$. To do this, we have to bound with high probability the quantity $\rho(E^TE)$. The gram matrix $E^TE$ contains all the scalar products between the columns of $E$ which are $E_{\cdot,j}\sim\mathcal{N}(0,\gamma^{2}S_{jj}I_{n})$. Thanks to \cref{Prop:VershyninRows}, we have on the set $\mathcal{A}_{\delta}$,
$$\rho(E^TE)\le \tilde{D}_{\delta}\gamma^{2}n\rho(S),$$
with $\tilde{D}_{\delta}$ explicit constant depending on $\delta$. We set $\Psi_{S_{2}}=\gamma^{2}n\rho(S)$ and $d_{\delta}=\tilde{D}_{\delta}$ to define the following assumption.
\paragraph{Assumption for Scenario 2.}\textit{The matrix $H^TH$  verifies $\frac{1}{4}\rho_{\min}(H^TH)>d_{\delta}n\gamma^{2}\rho(S)$, for $d_{\delta}$ depending only on $\delta$.} 

This property then applies for Scenario 1 where $\Psi_{S_{1}}=\gamma^{2}n$ with the following assumption.
\paragraph{Assumption for Scenario 1.}\textit{The matrix $H^TH$  verifies $\frac{1}{4}\rho_{\min}(H^TH)>d_{\delta}n\gamma^{2}$, for $d_{\delta}$ depending only on $\delta$.} 

These assumptions ensures the invertibility of $(\hat{H}^T\hat{H})^{-1}$ thanks to \cref{Prop:ConsequenceAssInverse}. 

\subsubsection{Scenario 3}
In Scenario 3 we have for $j\in\{1,...,K\},$ $E_{\cdot,j}\sim\mathcal{N}(0,\gamma^{2}A)$ iid. Unfortunately the rows of $E$ are not iid but using the fact that $\rho(EE^{T})=\rho(E^TE)$ we consider $EE^T$ as $(E^T)^TE^T$ where the rows of $E^T$ are iid. 
Considering $E^T\in\RR^{K\times n}$ we can apply the deviation bounds from \cref{Prop:VershyninRows} and get that, on $\mathcal{A}_{\delta}$,
$$|||\frac{1}{K}EE^T-\gamma^{2}A|||\le D_{\delta}\gamma^{2}\rho(A)\big(\frac{n}{K}\big).$$
Normalizing by $\frac{K}{n}$, we obtain the following
$$|||\frac{1}{n}EE^T-\gamma^{2}\frac{K}{n}A|||\le D_{\delta}\gamma^{2}\rho(A).$$
Hence on $\mathcal{A}_{\delta}$,
$$\rho(E^TE)\le  \tilde{D}_{\delta}n\gamma^{2}\rho(A),$$
with $\tilde{D}_{\delta}$ depending of $\delta$. To conclude, we set $d_{\delta}=\tilde{D}_{\delta}$ and $\Psi_{S_{3}}=n\gamma^{2}\rho(A)$.

\paragraph*{Assumption for Scenario 3.}\textit{The matrix $H^TH$  verifies $\frac{1}{4}\rho_{\min}(H^TH)>d_{\delta}n\gamma^{2}\rho(A)$, for $d_{\delta}$ depending only on $\delta$.}
\subsubsection{Scenario 4}\label{SubsubsecScen4}
We now define the function $\Psi_{S_{4}}$.
Thanks to deviation results (see \cref{Prop:DeviationResult}) on $\|E_{\cdot,j}\|^{2}$ we have for each $j\in\{1,...,K\}$,
$$\mathbb{P}\left(\|E_{\cdot,j}\|^{2}\ge g_{\delta}\gamma^{2}\mathrm{Tr}(A_{j}VA_{j}^{T})\right)\le \delta,$$ 
with $g_{\delta}$ depending on $\delta$. By regrouping the $K$ different events for each column of $E$ and replacing $\delta$ by $\frac{\delta}{K}$ we obtain an event $\mathcal{A}_{\delta,K}$ with $\mathbb{P}\left(\mathcal{A}_{\delta,K}\right)\ge 1-\delta$ where we get
\begin{align*}
\rho(E^TE)&\le \mathrm{Tr}(E^TE)\\
&\le\sum_{j=1}^{K}\|E_{\cdot,j}\|^{2}\\
&\le C_{\delta,K}\gamma^{2}\sum_{j=1}^{K}\mathrm{Tr}(A_{j}VA_{j}^{T}),
\end{align*}
with $C_{\delta,K}=c_{\delta}\ln(\frac{K}{\delta})$. We set $d_{\delta,K}=C_{\delta,K}$ and $\Psi_{S_{4}}=\gamma^{2}\sum_{j=1}^{K}\mathrm{Tr}(A_{j}VA_{j}^{T})$.
\paragraph*{Assumption for Scenario 4.}\textit{The matrix $H^TH$  verifies $$\frac{1}{4}\rho_{\min}(H^TH)>d_{\delta}\gamma^{2}\sum_{j=1}^{K}\mathrm{Tr}(A_{j}VA_{j}^{T}),$$ for $d_{\delta,K}$ depending only on $\delta$ and $\ln(\frac{K}{\delta})$.}

\paragraph*{}Now that the assumptions are made explicit, to conclude the proofs of the four theorems stated in \cref{Subsec:ProjOperatorEstimation}, we just need to bound the three mains terms $\mathrm{I},\mathrm{II}$ and $\mathrm{III}$ introduced in \cref{Th:ProjOperatorSkeletonProof} according to each Scenario. This part will be the focus of the next subsection.
\subsection[]{Bounds on the terms $\mathrm{I},\mathrm{II}$ and $\mathrm{III}$}
Thanks to \cref{Sec:ProjOperatorAdapt} we have set $d_{\delta,K}$ and $\Psi$ according to each scenario in order to apply \cref{Th:ProjOperatorSkeletonProof}. The final step is to upper bound the terms $\mathrm{I},\mathrm{II}$ and $\mathrm{III}$ for each scenario, taking the supremum on all $u$ such that $u^Tu=1$.  We propose a bound for the three terms for each scenario. Let us recall that terms $\mathrm{I}, \mathrm{II}$ and $\mathrm{III}$ are defined in \eqref{Eq:ProjOperatorI}, \eqref{Eq:ProjOperatorII} and \eqref{Eq:ProjOperatorIII}. 
\paragraph*{}All along this proof, we will use \cref{Prop:DeviationResult} several times (this number is independent of $n$ and $K$). For each use we will replace $\delta$ by $\frac{\delta}{2}$ in the event $\mathcal{A}_{\delta}$ in \cref{Def:EventA} in order to create a new event which will be the intersection of $\mathcal{A}_{\delta}$ and the one created by \cref{Prop:DeviationResult}. This allows to control simultaneously have the property of \cref{Prop:DeviationResult} and the inner property of $\mathcal{A}_{\delta}$ on $\rho(E^TE)$ with probability $1-\delta$. 

\paragraph*{The case of term $\mathrm{II}$} The term $\mathrm{II}$ is more complex than terms $\mathrm{I}$ and $\mathrm{III}$ and will be handled as follows. First, recall that $\Lambda_{u}=(H^TH)^{-1}H^Tu$. We have,
\begin{align*}
    \|(\hat{H}^T\hat{H}-H^TH)\Lambda_u\|^2&=\|(E^TH+H^TE+E^TE)\Lambda_u\|^2\\
    &\le2\|(E^TH+H^TE)\Lambda_u\|^2+2\|E^TE\Lambda_u\|^2\\
    &\le 4\|E^TH\Lambda_u\|^2+4\|H^TE\Lambda_u\|^2+2\|E^TE\Lambda_u\|^2.
\end{align*}
Hence, 
\begin{equation}\label{Eq:ProjOperatorDecompII}
\mathrm{II}\le 4\|E^TH\Lambda_u\|^2+4\|H^TE\Lambda_u\|^2+2\|E^TE\Lambda_u\|^2.
\end{equation}
Term $\mathrm{II}$ will be handled after these three terms have been bounded for each Scenario. 
\subsubsection{Scenario 1 and 2}\label{Subsub:Scen12}
We consider $E\in\RR^{n\times K}$ where $E_{i,\cdot}^{T}\sim\mathcal{N}(0,\gamma^{2}S)$ for all $i\in\{1,...,n\}$ with $S\in\RR^{K\times K}$. For Scenario 1, we consider $S=I_{K}$.
\subsubsection*{Scenario 1 and 2: Term I}
Let us recall that $\Lambda_{u}$ is defined in \eqref{Eq:ProjOperatorLambda}. 
We have $\EE[E\Lambda_u]=0$ with $(E\Lambda_{u})_i=\sum_{k=1}^{K}E_{ik}(\Lambda_u)_{k}$, every coordinate of $E\Lambda_{u}$ is Gaussian and the $(E_{i,\cdot})_{i=1}^{n}$ are independent. Then $E\Lambda_{u}$ is a Gaussian vector and
\begin{align*}
\EE[(E\Lambda_u)_i^{2}]&=\EE\bigg[(\sum_{k=1}^{K}E_{ik}(\Lambda_u)_k)^2\bigg]\\
&=\EE\big[\sum_{j,k}^{K}E_{ik}E_{ij}(\Lambda_u)_{k}(\Lambda_u)_{j}\big]\\
&=\gamma^{2}\sum_{j,k=1}^{K}\EE\big[E_{ij}E_{ik}(\Lambda_u)_{k}(\Lambda_{u})_{j}\big]\\
&=\gamma^{2}\Lambda_{u}^TS\Lambda_{u}.
\end{align*}
Then $E\Lambda_u\sim\mathcal{N}(0,\gamma^{2}(\Lambda_{u}^TS\Lambda_u)I_{n})$.
Thanks to the deviation result stated in \cref{Prop:DeviationResult}, we have on the set $\mathcal{A}_{\delta}$, 
$$\mathrm{I}\le C_{\delta}n\gamma^{2}\Lambda_u^TS\Lambda_u.$$
Then we consider the operator norm, 
\begin{align*}
    \underset{u\in\mathbb{R}^{n},u^Tu=1}{\sup}\mathrm{I}&\le \underset{u\in\mathbb{R}^{n},u^Tu=1}{\sup}C_{\delta}n\gamma^{2}\ \Lambda_{u}^TS\Lambda_{u}\\
    &\le \underset{u\in\mathbb{R}^{n},u^Tu=1}{\sup}C_{\delta}n\gamma^{2}u^TH(H^TH)^{-1}S(H^TH)^{-1}H^Tu\\
    &\le C_{\delta}n\gamma^{2}\rho\bigg(H(H^TH)^{-1}S(H^TH)^{-1}H^T\bigg)\\
    &\le C_{\delta}n\gamma^{2}\rho\bigg((H(H^TH)^{-1}S^{\frac{1}{2}})(H(H^TH)^{-1}S^{\frac{1}{2}})^T\bigg)\\
    &\le C_{\delta} n\gamma^{2}\rho\bigg(S^{\frac{1}{2}}(H^TH)^{-1}S^{\frac{1}{2}}\bigg)\\
    &\le C_{\delta}n\gamma^{2} \rho(S (H^TH)^{-1}).
\end{align*}
We deduce 
\begin{equation}\label{Eq:ProjOperatorS12TermI}
    \underset{u\in\RR^{n},\|u\|=1}{\sup}\mathrm{I}\le C_{\delta}n\gamma^{2}\rho(S(H^TH)^{-1}).
\end{equation}
\subsubsection*{Scenarios 1 and 2: Term $\mathrm{II}$}
The vectors $E^TH\Lambda_u$ and $H^TE\Lambda_u$ follow centered multivariate normal distributions. First, $$\mathrm{Var}(H^TE\Lambda_{u})=H^T\mathrm{Var}(E\Lambda_{u})H=\gamma^{2}\Lambda_{u}^TS\Lambda_{u}H^TH.$$ We have $(E^TH\Lambda_u)_i=E_{\cdot,i}^TH\Lambda_{u}$, We can then compute 
$$\EE[E_{\cdot,i}^TH\Lambda_uE_{\cdot,j}^TH\Lambda_{u}]=\gamma^{2}\Lambda_{u}^TH^TH\Lambda_{u}S_{ij}.$$
Hence we get
$\mathrm{Var}(E^TH\Lambda_u)=\gamma^{2}\|H\Lambda_u\|^2S$. 
However $E^TE\Lambda_u$ is not Gaussian.
Another way to choose for the last term is to consider $\|E^TE\Lambda_u\|^{2}$ as the following
\begin{align*}
    \|E^TE\Lambda_{u}\|^{2}&=\Lambda_{u}^{T}E^TEE^TE\Lambda_{u}\\
    &\le \rho(EE^T)\|E\Lambda_{u}\|^{2}\\
    &\le \rho(E^TE)\|E\Lambda_{u}\|^{2}\\
    &\le \tilde{D}_{\delta}\gamma^{2}n\rho(S)\|E\Lambda_{u}\|^{2}\\
    &\le \frac{1}{4}\frac{\tilde{D}_{\delta}}{d_{\delta}}\rho_{\min}(H^TH)\|E\Lambda_{u}\|^{2}\\
    &\le \frac{1}{4}\frac{C_{\delta}\tilde{D}_{\delta}}{d_{\delta}}\rho_{\min}(H^TH)\gamma^{2}n\Lambda_{u}^TS\Lambda_{u}.
\end{align*}
Where we use the fact that $\rho(EE^T)=\rho(E^TE)$ and where the last step results from the deviation result of \cref{Prop:DeviationResult} applied with \textbf{Assumption E.1} in Scenario 2. By this way, we keep the quantity $\Lambda_{u}^TS\Lambda_{u}$ which also appears in the first term $\mathrm{I}$. Combining the three terms thanks to the deviation result of \cref{Prop:DeviationResult}, we get the following inequalities on the set $\mathcal{A}_{\delta}$:
\begin{align*}
    \|H^TE\Lambda_{u}\|^2&\le C_{\delta}\gamma^{2}\Lambda_{u}^TS\Lambda_{u}\mathrm{Tr}(H^TH),\\
    \|E^TH\Lambda_{u}\|^{2}&\le C_{\delta}\gamma^{2}\|H\Lambda_{u}\|^{2}\mathrm{Tr}(S),\\
    \|E^TE\Lambda_{u}\|^{2}&\le \frac{1}{4}\frac{C_{\delta}\tilde{D}_{\delta}}{d_{\delta}}\gamma^{2}n\rho_{\min}(H^TH)\Lambda_{u}^TS\Lambda_{u}.
\end{align*}
Finally, thanks to the decomposition \eqref{Eq:ProjOperatorDecompII}, we get
$$\|(\hat{H}^T\hat{H}-H^TH)\Lambda_{u}\|^{2}\le C_{\delta}\gamma^{2}\bigg(\Lambda_{u}^TS\Lambda_{u}\mathrm{Tr}(H^TH)+\|H\Lambda_{u}\|^{2}\mathrm{Tr}(S)+\frac{1}{4}\frac{D_{\delta}}{d_{\delta}}n\rho_{\min}(H^TH)\Lambda_{u}^TS\Lambda_{u}\bigg).$$

We now consider the supremum on $\mathrm{II}$.
We focus on the term $\rho_{\min}(H^TH)^{-1}\mathrm{II}$. We have, 
\begin{align*}
    \rho_{\min}(H^TH)^{-1}\mathrm{II}&=\rho_{\min}(H^TH)^{-1}\|(\hat{H}^T\hat{H}-H^TH)\Lambda_{u}\|^{2}\\
    &\le \frac{C_{\delta}\gamma^{2}}{\rho_{\min}(H^TH)}\bigg(\|H\Lambda_{u}\|^{2}\mathrm{Tr}(S)+\left(\frac{1}{4}\frac{\tilde{D}_{\delta}}{d_{\delta}}n\rho_{\min}(H^TH)+\mathrm{Tr}(H^TH)\right)\Lambda_{u}^TS\Lambda_{u}\bigg)\\
    &\le C_{\delta}\gamma^{2}\left(\Lambda_{u}^TS\Lambda_{u}\frac{\mathrm{Tr}(H^TH)}{\rho_{\min}(H^TH)}+\|H\Lambda_{u}\|^{2}\frac{\mathrm{Tr}(S)}{\rho_{\min}(H^TH)}+\frac{1}{4}\frac{\tilde{D}_{\delta}}{d_{\delta}}n\Lambda_{u}^TS\Lambda_{u}\right)\\
    &\le C_{\delta}\gamma^{2}\left(\mathrm{Cond}(H^TH)K\Lambda_{u}^TS\Lambda_{u}+\|H\Lambda_{u}\|^{2}\frac{\mathrm{Tr}(S)}{\rho_{\min}(H^TH)}+\frac{1}{4}\frac{\tilde{D}_{\delta}}{d_{\delta}}n\Lambda_{u}^TS\Lambda_{u}\right).
\end{align*}
We deduce 
\begin{multline}
\underset{u\in\RR^{n},\|u\|=1}{\sup} \big(\rho_{\min}(H^TH)^{-1}\mathrm{II}\big)\le C_{\delta}\gamma^{2}\bigg(K\rho(S(H^TH)^{-1})\mathrm{Cond}(H^TH)\\+\frac{\mathrm{Tr}(S)}{\rho_{\min}(H^TH)}+\frac{1}{4}\frac{\tilde{D}_{\delta}}{d_{\delta}}n\rho(S(H^TH)^{-1})\bigg).\label{Eq:ProjOperatorS12TermII}
\end{multline}
\subsubsection*{Scenario 1 and 2: Term $\mathrm{III}$}
We focus for $u\in\RR^{n}$ on $\mathrm{III}=\|H(H^TH)^{-1}E^{T}u\|^2$. The vector $H(H^TH)^{-1}E^Tu$ is centered and satisfies 
\begin{align*}
\mathrm{Var}\big(H(H^TH)^{-1}E^Tu\big)&=H(H^TH)^{-1}\mathrm{Var}\big(E^Tu\big)(H^TH)^{-1}H^T\\
&=\gamma^{2}\|u\|^2H(H^TH)^{-1}S(H^TH)^{-1}H^T.
\end{align*}
Thanks to the deviation result of \cref{Prop:DeviationResult} we get 
$$\|H(H^TH)^{-1}E^Tu\|^2\le C_{\delta}\gamma^{2}\|u\|^2\mathrm{Tr}\big((H^TH)^{-1}S\big).$$
The matrix $(H^TH)^{-1}$ is positive definite. Hence, there exists $P$ such that $(H^TH)^{-1}=PDP^{T}$ with $PP^T=I_{K}$ and $D$ positive diagonal matrix. Then we have 
$$\mathrm{Tr}\big((H^TH)^{-1}S\big)\le \mathrm{Tr}\big(PDP^{T}S\big)\le \mathrm{Tr}\big(DP^{T}S P\big)\le \underset{1\le i\le K}{\max}(D_{i}) \mathrm{Tr}(S)\le \frac{\mathrm{Tr}(S)}{\rho_{\min}(H^TH)}.$$
Actually, following the previous inequalities, we can prove that for any symmetric positive matrices $M$ and $N$ we have 
\begin{equation}\label{Eq:ProjOperatorTraceTrick}
\mathrm{Tr}(MN)\le \rho(M)\mathrm{Tr}(N).
\end{equation}
This result result will be used various times in the following.
By taking the supremum on all normalized $u$, we get 
\begin{align*}
\mathrm{III}&=\|H(H^TH)^{-1}E^Tu\|^{2}
&\le C_{\delta}\gamma^{2}\mathrm{Tr}\bigg((H^TH)^{-1}S\bigg)\\
&\le C_{\delta}\gamma^{2}\rho_{\min}(H^TH)^{-1}\mathrm{Tr}\big(S\big).
\end{align*}
We deduce 
\begin{equation}
    \underset{u\in\RR^{n},\|u\|=1}{\sup}\mathrm{III}\le C_{\delta}\gamma^{2}\rho_{\min}(H^TH)^{-1}\mathrm{Tr}(S).\label{Eq:ProjOperatorS12TermIII}
\end{equation}

This concludes the proof of \cref{Th:ProjOperatorRows} and \cref{Th:ProjOperatorIndependent}, by setting $S=I_{K}$ and by using \eqref{Eq:ProjOperatorS12TermI}, \eqref{Eq:ProjOperatorS12TermII} and \eqref{Eq:ProjOperatorS12TermIII} and the decomposition in \cref{Th:ProjOperatorSkeletonProof}.
We now perform the same analysis for Scenario 3.

\subsubsection{Scenario 3}
Now we focus on Scenario 3 where the columns of $E$ are independent. We recall that $\Lambda_{u}$ is defined in \eqref{Eq:ProjOperatorLambda} and $\mathrm{I}, \mathrm{II}$ and $\mathrm{III}$ are given in \eqref{Eq:ProjOperatorI}, \eqref{Eq:ProjOperatorII} and \eqref{Eq:ProjOperatorIII}.
\subsubsection*{Scenario 3: Term $\mathrm{I}$}
We keep the same notations as for Scenario 2. 
Under Scenario 3, we have
$E\Lambda_{u}\sim\mathcal{N}(0,\gamma^{2}\|\Lambda_{u}\|^{2}A)$.
Thanks to the deviation result of \cref{Prop:DeviationResult} we get $$\|(\hat{H}-H)(H^TH)^{-1}H^Tu\|^{2}\le C_{\delta}\gamma^{2}\mathrm{Tr}(A)\|\Lambda_{u}\|^{2}.$$
By taking all normalized $u\in\RR^{n}$
we get 
\begin{equation}\label{Eq:ProjOperatorS3TermI}
\underset{u\in\RR^{n}, \|u\|=1}{\sup}\mathrm{I}\le C_{\delta}\gamma^{2}\frac{\mathrm{Tr}(A)}{\rho_{\min}(H^TH)}.
\end{equation}
\subsubsection*{Scenario 3: Term $\mathrm{II}$}
We use the decomposition of $\mathrm{II}$ in three terms seen previously in \cref{Eq:ProjOperatorDecompII}. 
For the last term we have,
$$\|E^TE\Lambda_{u}\|^{2}\le \rho(EE^{T})\|E\Lambda_{u}\|^{2}\le \rho(E^TE)C_{\delta}\gamma^{2}\|\Lambda_{u}\|^{2}\mathrm{Tr}(A).$$
Hence, $$\|E^TE\Lambda_{u}\|^{2}\le \frac{C_{\delta}\tilde{D}_{\delta}}{d_{\delta}}\frac{\rho_{\min}(H^TH)}{4}\gamma^{2}\|\Lambda_{u}\|^{2}\mathrm{Tr}(A).$$
We now focus on terms $\|E^TH\Lambda_{u}\|^{2}$ and $\|H^TE\Lambda_{u}\|^{2}$.
For $H^TE\Lambda_{u}$, we have 
$$\mathrm{Var}(H^TE\Lambda_{u})=H^T\mathrm{Var}(E\Lambda_{u})H=\gamma^{2}H^TAH\|\Lambda_{u}\|^{2}.$$
Thanks to the deviation result of \cref{Prop:DeviationResult}, we deduce that
$$\|H^TE\Lambda_{u}\|^{2}\le C_{\delta}\gamma^{2}\|\Lambda_{u}\|^{2}\mathrm{Tr}(H^TAH).$$
Next, we have $(E^TH\Lambda_{u})_{i}=E_{\cdot,i}^TH\Lambda_{u}$. Then 
$$\EE[\Lambda_{u}^TH^T E_{\cdot,i}E_{\cdot,j}^{T}H\Lambda_{u}]=\gamma^{2}\delta_{ij}\Lambda_{u}^TH^TAH\Lambda_{u}.$$
Hence, $\mathrm{Var}(E^TH\Lambda_{u})=\gamma^{2}\Lambda_{u}^TH^TAH\Lambda_{u}I_{K},$ and
$$\|E^TH\Lambda_{u}\|^{2}\le C_{\delta}\gamma^{2}K\Lambda_{u}^{T}H^TAH\Lambda_{u}.$$
Combining these three different terms, we get that, on $\mathcal{A}_{\delta},$
\begin{align*}
    \|E^TE\Lambda_{u}\|^{2}&\le\frac{C_{\delta}\tilde{D}_{\delta}}{d_{\delta}}\frac{\rho_{\min}(H^TH)}{4}\gamma^{2}\|\Lambda_{u}\|^{2}\mathrm{Tr}(A),\\
    \|H^TE\Lambda_{u}\|^{2}&\le C_{\delta}\gamma^{2}\|\Lambda_{u}\|^{2}\mathrm{Tr}(H^TAH),\\
    \|E^TH\Lambda_{u}\|^{2}&\le C_{\delta}\gamma^{2}K\Lambda_{u}^TH^TAH\Lambda_{u}.
\end{align*}
Finally, on $\mathcal{A}_{\delta}$,
$$\mathrm{II}\le C_{\delta}\gamma^{2}\bigg(K\Lambda_{u}^TH^TAH\Lambda_{u}+\|\Lambda_{u}\|^{2}\mathrm{Tr}(H^TAH)+\frac{\tilde{D}_{\delta}}{d_{\delta}}\frac{\rho_{\min}(H^TH)}{4}\|\Lambda_{u}\|^{2}\mathrm{Tr}(A)\bigg).$$
Let us study $\rho_{\min}(H^TH)^{-1}\mathrm{II}$ which appears in \cref{Th:ProjOperatorCollumns}. We have
\begin{align}\label{Eq:ProjOperatoraprime}
    \rho_{\min}(H^TH)^{-1}\mathrm{II}&\le\bigg(C_{\delta}\gamma^{2}K\frac{\rho\big(A\big)}{\rho_{\min}(H^TH)}+C_{\delta}\gamma^{2}\frac{\mathrm{Tr}(H^TAH)}{\rho_{\min}(H^TH)^{2}}+\frac{C_{\delta}\tilde{D}_{\delta}}{d_{\delta}}\frac{1}{4}\gamma^{2}\frac{\mathrm{Tr}(A)}{\rho_{\min}(H^TH)}\bigg),
\end{align}
where we used the equality
$$\rho(H(H^TH)^{-1}H^TAH(H^TH)^{-1}H^T)=\rho\big((A^{\frac{1}{2}}H(H^TH)^{-1}H^{T}A^{\frac{1}{2}})\big).$$
We clearly have $$\rho\big(A^{\frac{1}{2}}H(H^TH)^{-1}H^TA^{\frac{1}{2}}\big)\le \rho(A)\rho\big(H(H^TH)^{-1}H^T\big)\le \rho(A).$$
Then, for the second term with $HH^T=P^TD P$ with $A$ symmetric positive we have thanks to \eqref{Eq:ProjOperatorTraceTrick}, 
\begin{equation}\label{Eq:ProjOperatorbprime}
\mathrm{Tr}(H^TAH)=\mathrm{Tr}(HH^TA)\le\rho(HH^T)\mathrm{Tr}(A).
\end{equation}
We conclude from \eqref{Eq:ProjOperatoraprime} and \eqref{Eq:ProjOperatorbprime}, on $\mathcal{A}_{\delta}$
\begin{equation}\label{Eq:ProjOperatorS3TermII}
\underset{u\in\mathbb{R}^{n},\|u\|=1}{\sup}\rho_{\min}(H^TH)^{-1}\mathrm{II}\le F_{\delta}\gamma^{2}\mathrm{Cond}(H^TH)\max\bigg(K\frac{\rho(A)}{\rho_{\min}(H^TH)},\frac{\mathrm{Tr}(A)}{\rho_{\min}(H^TH)}\bigg),
\end{equation}
where $F_{\delta}$ is a constant depending on $d_{\delta},\tilde{D}_{\delta}$ and $C_{\delta}$.
\subsubsection*{Scenario 3: Term $\mathrm{III}$}
We focus on the term $\mathrm{III}=\|H(H^TH)^{-1}E^Tu\|^{2}$, the associated variance matrix is $$H(H^TH)^{-1}\mathrm{Var}(E^Tu)(H^TH)^{-1}H^T=H(H^TH)^{-1}\gamma^{2}u^TAu(H^TH)^{-1}H^T.$$
Thanks to the deviation result given in \cref{Prop:DeviationResult}, we deduce
$$\mathrm{III}\le C_{\delta}\gamma^{2}u^TAu\mathrm{Tr}((H^TH)^{-1}).$$
Taking the supremum on all normalized $u\in\RR^{n}$, on $\mathcal{A}_{\delta}$,
\begin{equation}\label{Eq:ProjOperatorS3TermIII}
\underset{u\in\RR^{n},\|u\|=1}{\sup} \mathrm{III} \le C_{\delta}\gamma^{2}\rho(A)\mathrm{Tr}\big((H^TH)^{-1}\big).
\end{equation}
This concludes the proof of \cref{Th:ProjOperatorCollumns} by using \eqref{Eq:ProjOperatorS3TermI}, \eqref{Eq:ProjOperatorS3TermII} and \eqref{Eq:ProjOperatorS3TermIII} and the decomposition of \cref{Th:ProjOperatorSkeletonProof}.
\subsubsection{Scenario 4}\label{Subsub:Scen4}
We recall that $\Lambda_{u}$ is defined in \eqref{Eq:ProjOperatorLambda} and $\mathrm{I}, \mathrm{II}$ and $\mathrm{III}$ are given in \eqref{Eq:ProjOperatorI}, \eqref{Eq:ProjOperatorII} and \eqref{Eq:ProjOperatorIII}.
\subsubsection*{Scenario 4: Term $\mathrm{I}$}
We consider Scenario 4 where for all $j\in\{1,...,K\}$, $E_{\cdot,j}\sim\mathcal{N}(0,A_{j}VA_{j}^{T})$. We have
\begin{align*}
    E\Lambda_{u}&=\sum_{j=1}^{K}E_{\cdot,j}(\Lambda_{u})_{j}\\
    &= \sum_{j=1}^{K}A_{j}(\hat{v}-v)(\Lambda_{u})_{j}\\
    &= \bigg(\sum_{j=1}^{K}A_{j}(\Lambda_{u})_{j}\bigg)(\hat{v}-v).
\end{align*}
Hence, 
$$\mathrm{Var}(E\Lambda_{u})=\gamma^{2}\sum_{l,m=1}^{K}(\Lambda_{u})_{l}(\Lambda_{u})_{m}A_{l}VA_{m}^{T}.$$
Thanks to the deviation result of \cref{Prop:DeviationResult}, we obtain that on $\mathcal{A}_{\delta}$,
\begin{align*}
\|E\Lambda_{u}\|^{2}&\le C_{\delta,K}\gamma^{2}\sum_{l,m=1}^{K}(\Lambda_{u})_{m}(\Lambda_{u})_{l}\mathrm{Tr}\big(A_{l}VA_{m}^{T}\big)\\
&\le C_{\delta,K}\gamma^{2}\bigg(\sum_{l=1}^{K}(\Lambda_{u})_{l}\sqrt{\mathrm{Tr}(A_{l}^TVA_{l})}\bigg)^{2}\\ 
&\le \gamma^{2}C_{\delta,K}\|\Lambda_{u}\|^{2}\sum_{l=1}^{K}\mathrm{Tr}(A_{l}^TVA_{l}),
\end{align*}
where $C_{\delta,K}$ is depending only on $\delta$ and $\ln(\frac{K}{\delta})$. Thus, on $\mathcal{A}_{\delta}$,
$$\mathrm{I}\le \gamma^{2}C_{\delta,K}\|\Lambda_{u}\|^{2}\sum_{l=1}^{K}\mathrm{Tr}(A_{l}^TA_{l}V).$$
We take the supremum on all normalized $u\in\RR^{n}$ to get on $\mathcal{A}_{\delta}$
\begin{equation}\label{Eq:ProjOperatorS4TermI}
\underset{u\in\mathbb{R}^{n}, u^Tu=1}{\sup}\mathrm{I}\le C_{\delta,K}\gamma^{2}\frac{\sum_{l=1}^{K}\mathrm{Tr}(A_{l}^{T}A_{l}V)}{\rho_{\min}(H^TH)}.
\end{equation}
\subsubsection*{Scenario 4: Term $\mathrm{II}$}
We recall the decomposition \eqref{Eq:ProjOperatorDecompII} of $\mathrm{II}$ in three terms,
\begin{equation}\label{Eq:ProjOperatoraseconde}
\mathrm{II}\le 4\|E^TH\Lambda_u\|^2+4\|H^TE\Lambda_u\|^2+2\|E^TE\Lambda_u\|^2.
\end{equation}
\paragraph*{Term $\|E^TE\Lambda_{u}\|^{2}$.}Under Scenario 4, on $\mathcal{A}_{\delta}$, 
\begin{align}
\|E^TE\Lambda_{u}\|^{2}&\le \rho(EE^T)\|E\Lambda_{u}\|^{2}\nonumber\\
&\le \rho(E^TE) C_{\delta,K}\gamma^{2}\sum_{l=1}^{K}\mathrm{Tr}\big(A_{l}VA_{l^{T}}\big)\nonumber\\
&\le \frac{C_{\delta,K}}{d_{\delta,K}}\frac{\rho_{\min}(H^TH)}{4}\|\Lambda_{u}\|^{2}\sum_{l=1}^{K}\mathrm{Tr}\big(A_{l}VA_{l}^{T}\big).\label{Eq:ProjOperatorbseconde}
\end{align}

\paragraph*{Term $\|H^TE\Lambda_{u}\|^{2}$.}Using $\mathrm{Var}(H^TE\Lambda_{u})=H^T\mathrm{Var}(E\Lambda_{u})H$, we have, on $\mathcal{A}_{\delta}$,
\begin{align*}
\|H^TE\Lambda_{u}\|^{2}&\le C_{\delta,K}\gamma^{2}\sum_{l,m=1}^{K}(\Lambda_{u})_{l}(\Lambda_{u})_{m}\mathrm{Tr}(HH^TA_{l}VA_{m}^{T})\\
&\le C_{\delta,K}\gamma^{2}\|\Lambda_{u}\|^{2}\sum_{l=1}^{K}\mathrm{Tr}(HH^TA_{l}VA_{l}^{T}).
\end{align*}

Using the fact that $HH^{T}$ and $A_{l}VA_{l}^{T}$ are positive symmetric matrices, thanks to \eqref{Eq:ProjOperatorTraceTrick} we get, on $\mathcal{A}_{\delta}$,
\begin{equation}\label{Eq:ProjOperatorcseconde}
\|H^TE\Lambda_{u}\|^{2}\le C_{\delta,K}\gamma^{2}\rho(H^TH)\|\Lambda_{u}\|^{2}\sum_{l=1}^{K}\mathrm{Tr}(A_{l}VA_{l}^{T}).
\end{equation}

\paragraph*{Term $\|E^TH\Lambda_{u}\|^{2}$.}
We have 
\begin{align*}
\EE[(E^TH\Lambda_{u})_{i}(E^TH\Lambda_{u})_{j}]&=\EE[\Lambda_{u}^{T}H^TE_{\cdot,i}E_{\cdot,j}^{T}H\Lambda_{u}]\\
&=\Lambda_{u}^TH^T\EE[E_{\cdot,i}E_{\cdot,j}^{T}]H\Lambda_{u}\\
&=\Lambda_{u}^TH^TA_{i}\EE[(\hat{v}-v)(\hat{v}-v)^T]A_{j}^{T}H\Lambda_{u}\\
&=\Lambda_{u}^TH^TA_{i}\EE[(\hat{v}-v)(\hat{v}-v)^T]A_{j}^{T}H\Lambda_{u}\\
&=\gamma^{2}\Lambda_{u}^TH^TA_{i}VA_{j}^{T}H\Lambda_{u}.
\end{align*}
Applying the deviation result of \cref{Prop:DeviationResult} we have, on $\mathcal{A}_{\delta}$,
$$\|E^TH\Lambda_{u}\|^{2}\le C_{\delta,K}\gamma^{2}\sum_{i=1}^{K}\Lambda_{u}^TH^TA_{i}VA_{i}^TH\Lambda_{u}.$$
Then, we obtain 
\begin{align*}
\Lambda_{u}^TH^TA_{i}VA_{i}^{T}H\Lambda_{u}&=\mathrm{Tr}(H\Lambda_{u}\Lambda_{u}^{T}H^{T}A_{i}VA_{i}^{T})\\
&\le \rho(H\Lambda_{u}\Lambda_{u}^TH^T)\mathrm{Tr}(A_{i}VA_{i}^{T})\\
&\le \|H\Lambda_{u}\|^{2}\mathrm{Tr}(A_{i}VA_{i}^{T}),
\end{align*}
where we used the fact that $H\Lambda_{u}\Lambda_{u}^{T}H^T$ and $A_{i}VA_{i}^{T}$ are symmetric positive matrices in order to use \cref{Eq:ProjOperatorTraceTrick}.
Hence, 
\begin{equation}\label{Eq:ProjOperatordseconde}
\|E^TH\Lambda_{u}\|^{2}\le C_{\delta,K}\gamma^{2}\|H\Lambda_{u}\|^{2}\sum_{l=1}^{K}\mathrm{Tr}(A_{i}VA_{i}^{T}).
\end{equation}

We consider $\rho_{\min}(H^TH)^{-1}\mathrm{II}$. Using decomposition \eqref{Eq:ProjOperatoraseconde} and inequalities \eqref{Eq:ProjOperatorbseconde}, \eqref{Eq:ProjOperatorcseconde} and \eqref{Eq:ProjOperatordseconde}, we obtain, on $\mathcal{A}_{\delta}$,
\begin{equation}\label{Eq:ProjOperatorS4TermII}
\underset{u\in\mathbb{R}^{n}, u^Tu=1}{\sup}\rho_{\min}(H^TH)^{-1}\mathrm{II}\le \bigg(C_{\delta,K}\big(1+\mathrm{Cond}(H^TH)+\frac{1}{d_{\delta,K}}\big)\gamma^{2}\frac{\sum_{l=1}^{K}\mathrm{Tr}(A_{l}^TA_{l}V)}{\rho_{\min}(H^TH)}\bigg).
\end{equation}
\subsubsection*{Scenario 4: Term $\mathrm{III}$}
We focus on Term $\|H(H^TH)^{-1}E^Tu\|^{2}$. The associated variance matrix is $$\mathrm{Var}(H(H^TH)^{-1}E^Tu)=H(H^TH)^{-1}\mathrm{Var}(E^Tu)(H^TH)^{-1}H^T.$$
Then, $\mathrm{Tr}(\mathrm{Var}(H(H^TH)^{-1}E^Tu))=\mathrm{Tr}\big((H^TH)^{-1}B\big)$ where $B_{ij}=\gamma^{2}u^TA_{i}VA_{j}^Tu$.
Thanks to \cref{Eq:ProjOperatorTraceTrick}, we obtain
$$\mathrm{Tr}\big((H^TH)^{-1}B\big)\le \rho_{\min}(H^TH)^{-1}\mathrm{Tr}(B).$$
It yields
$$\mathrm{III}\le C_{\delta,K}\gamma^{2}\rho_{min}(H^TH)^{-1}\|u\|^{2}\sum_{l=1}^{K}\mathrm{Tr}(A_{i}VA_{i}^{T}).$$
Considering $u\in\mathbb{R}^{n}$ such that $u^Tu=1$, we get on $\mathcal{A}_{\delta}$
\begin{equation}\label{Eq:ProjOperatorS4TermIII}
\underset{u\in\mathbb{R}^{n}, u^Tu=1}{\sup}\mathrm{III}\le C_{\delta,K}\gamma^{2}\frac{\sum_{l=1}^{K}\mathrm{Tr}(A_{l}^TA_{l}V)}{\rho_{\min}(H^TH)}.
\end{equation}
This concludes the proof of \cref{Th:ProjOperatorGeneralized} using the decomposition of \cref{Th:ProjOperatorSkeletonProof} and inequalities \eqref{Eq:ProjOperatorS4TermI}, \eqref{Eq:ProjOperatorS4TermII} and \eqref{Eq:ProjOperatorS4TermIII}.

\section{Proof for the regularized versions}\label{Sec:ProjOperatorRegularizedProof}
As in Appendix \ref{Sec:ProjOperatorSkeleton} we propose a general proof scheme for \cref{Th:ProjOperatorIndependentPenalinise}, \cref{Th:ProjOperatorRowsPenalinise}, \cref{Th:ProjOperatorColumnsPenalinise} and \cref{Th:ProjOperatorGeneralizedPenalinise} for the regularized versions of projection operators.  

\subsection{Technical results}
We recall that for $\Delta_{\alpha}=\alpha I_{K}$ with $(H^TH)_{\alpha}=(H^TH+\Delta_{\alpha})$. We have the following result.
\begin{lem}\label{Lemma:ProjOperatorAlpha}
For $\alpha\ge0$ and $(H^TH)_{\alpha}=H^TH+\Delta_{\alpha}$, we have 
$$\rho_{\min}\left((H^TH)_{\alpha}\right)\ge \rho_{\min}(H^TH)+\alpha.$$
\end{lem}
\begin{definition}\label{Def:EventB}
Let $\delta\in(0,1)$. We define an event $\mathcal{B}_{\delta}$ satisfying $\mathbb{P}(\mathcal{B}_{\delta})\ge 1-\delta$, 
on which we have $\sum_{j=1}^{K}\|E_{\cdot,j}\|^{2}\le d_{\delta,K}\Psi$ where $d_{\delta,K}$ depends on $\delta$ and $K$, and where $\Psi\ge0$. 
\end{definition}
The parameters $d_{\delta,K}$ and $\Psi$ remain unchanged in the regularized versions. They will be equal to those defined in \cref{Sec:ProjOperatorAdapt} according to the scenario.
\begin{proposition}\label{Prop:EigenControl}
Let $\Delta_{\alpha}=\alpha I_{K}$ with $\alpha=2d_{\delta,K}\Psi$. We have, on the event $\mathcal{B}_{\delta}$,
\begin{align}
    \rho_{\min}\big((\hat{H}^T\hat{H})_{\alpha}\big)&\ge \frac{\rho_{\min}(H^TH)}{2},\label{ProjOperatorRa}\\
    \rho\big((\hat{H}^T\hat{H})_{\alpha}-(H^TH)_{\alpha}\big)&\le 
    2(\alpha+\rho(H^TH)),\label{Eq:ProjOperatorRb}\\
    \rho(E^TE)&\le \frac{\alpha}{2},\label{Eq:ProjOperatorRc}\\
    \frac{\rho\big((\hat{H}^T\hat{H})_{\alpha}-(H^TH)_{\alpha}\big)^{2}}{\rho_{\min}\big((H^TH)_{\alpha}\big)^{2}}&\le 32\mathrm{Cond}(H^TH)^{2}.\label{Eq:ProjOperatorRd}
\end{align}
\end{proposition}
\begin{proof}
First,
\begin{align*}
    x^T(\hat{H}^T\hat{H})_{\alpha}x&=x^TH^THx+x^T\Delta_{\alpha}x+x^TE^TEx+2x^T(\hat{H}-H)^THx\\
    &\ge x^TH^THx+x^T\Delta_{\alpha} x-2|x^T(\hat{H}-H)^THx|\\
    &\ge \frac{1}{2}x^TH^THx+x^T\Delta_{\alpha}x-2x^TE^TEx\\
    &\ge \frac{\rho_{\min}(H^TH)}{2}+x^T\bigg(\Delta_{\alpha}-2E^TE\bigg)x.
\end{align*}
We can use the following inequality:
\begin{align*}
    2x^TE^TEx&=2\sum_{i,j=1}^{K}x_{i}x_{j}E_{\cdot,i}^TE_{\cdot,j}\\
    &\le2\sum_{i,j=1}^{K}|x_{i}||x_{j}|\|E_{\cdot,i}\|\|E_{\cdot,j}\|\\
    &\le 2\bigg(\sum_{i=1}^{K}|x_{i}|\|E_{\cdot,i}\|\bigg)^{2}\\
    &\le 2\sum_{i=1}^{K}\|E_{\cdot,i}\|^{2}\\
    &\le 2d_{\delta,K}\Psi.
\end{align*}
If $\alpha_i=2d_{\delta,K}\Psi$ it ensures the positivity of $x^T\bigg(\Delta_{\alpha}-2E^TE\bigg)x$ and consequently on the event $\mathcal{B}_{\delta}$, we deduce inequality \eqref{ProjOperatorRa} and inequality \eqref{Eq:ProjOperatorRc}
For the inequality \eqref{Eq:ProjOperatorRb}, we have 
\begin{align*}
    x^T\big((\hat{H}^T\hat{H})_{\alpha}-(H^TH)_{\alpha}\big)x&\le  x^T(E^TH+H^TE)x+x^TE^TEx\\
    &\le (1+\frac{1}{2})x^TE^TEx+ 2x^TH^THx\\
    &\le 2(\alpha + \rho(H^TH)). 
\end{align*}
Hence we finally get inequality \eqref{Eq:ProjOperatorRd} by
\begin{align*}
    \frac{\rho\big((\hat{H}^T\hat{H})_{\alpha}-(H^TH)_{\alpha}\big)^{2}}{\rho_{\min}\big((H^TH)_{\alpha}\big)^{2}}&\le 4 \frac{(\alpha+\rho(H^TH))^{2}}{\big(\alpha+\rho_{\min}(H^TH)\big)^2}\\
    &\le 8 \frac{\alpha^{2}+\rho(H^TH)^{2}}{\big(\alpha+\rho_{\min}(H^TH)\big)^{2}}\\
    &\le 8\big(1+\mathrm{Cond}(H^TH))^{2}\\
    &\le 32 \mathrm{Cond}(H^TH)^{2}.
\end{align*}
This concludes the proof.
\end{proof}
\subsection{Bias induced by regularization}\label{Subsec:ProjOperatorBiasRegularized}
The regularization induces a natural bias in the estimation of the projection operator. 
We use the following decomposition:
$$P_{[H]}-P_{[\hat{H}]_{\alpha}}=P_{[H]}-P_{[H]_{\alpha}}+P_{[H]_{\alpha}}-P_{[\hat{H}]_{\alpha}},$$
where $P_{[H]}=H(H^TH)^{-1}H^T, P_{[H]_{\alpha}}=H(H^TH)_{\alpha}^{-1}H^T$ and $P_{[\hat{H}]_{\alpha}}=\hat{H}(\hat{H}^T\hat{H})_{\alpha}^{-1}\hat{H}^T$.
We focus on $|||P_{[H]}-P_{[H]_{\alpha}}|||^{2}$ with $P_{[H]_{\alpha}}=H(H^TH)_{\alpha}^{-1}H^T$. Let $u\in\RR^{n}$ such that $u^Tu=1$, we define $\Lambda_{u}=(H^TH)^{-1}H^Tu$.
\begin{proposition}
Let $\alpha\in \RR_{+},$ we have 
$$|||P_{[H]_{\alpha}}-P_{[H]}|||^{2}\le \frac{\alpha}{\rho_{\min}(H^TH)}.$$
\end{proposition}
\begin{proof} Let $u\in\mathbb{R}^{n},$ such that $u^Tu=1$. We have
\begin{align*}
    \|(P_{[H]_{\alpha}}(u)-P_{[H]})(u)\|^{2}&\le \|H\big((H^TH)_{\alpha}^{-1}-(H^TH)^{-1}\big)H^Tu\|^{2}\\
    &\le \|H(H^TH)_{\alpha}^{-1}\big((H^TH)_{\alpha}-(H^TH)\big)(H^TH)^{-1}H^Tu\|^{2}\\
    &\le \|H(H^TH)_{\alpha}^{-1}\Delta_{\alpha}\Lambda_{u}\|^{2}\\
    &\le \alpha^{2}\|H(H^TH)_{\alpha}^{-1}\Lambda_{u}\|^{2}\\
    &\le \alpha^{2}\rho\bigg((H^TH)_{\alpha}H^TH(H^TH)_{\alpha}^{-1}\bigg)\|\Lambda_{u}\|^{2}\\
    &\le\alpha^{2}\rho\bigg((H^TH)_{\alpha}^{-1}H^TH\bigg)\rho_{\min}\big((H^TH)_{\alpha}\big)^{-1}\|\Lambda_{u}\|^{2}\\
    &\le \alpha\|\Lambda_{u}\|^{2}\\
    &\le \frac{\alpha}{\rho_{\min}(H^TH)}.
\end{align*}
Where the last step results from taking the $\sup$ on all $u$ with $u^Tu=1$.
\end{proof}
\subsection{Main decomposition for the regularized theorems}
We introduce the three mains terms which will bound the desired quantity thanks to our decomposition:
\begin{align}
    \mathrm{I}_{\alpha}&=u^TH(H^TH)_{\alpha}^{-1}(\hat{H}-H)^T(\hat{H}-H)(H^TH)_{\alpha}^{-1}H^Tu,\label{Eq:ProjOperatorTermIAlph}\\
    \mathrm{II}_{\alpha}&=u^TH(H^TH)_{\alpha}^{-1}\left((\hat{H}^T\hat{H})_{\alpha}-(H^TH)_{\alpha}\right)^{2}(H^TH)_{\alpha}^{-1}H^Tu,\label{Eq:ProjOperatorTermIIAlph}\\
    \mathrm{III}_{\alpha}&=(u^TH-u^T\hat{H})(H^TH)_{\alpha}^{-1}(H^Tu-\hat{H}^Tu).\label{Eq:ProjOperatorTermIIIAlph}
\end{align}
The first term corresponds to a "linear term" where we consider that regularized coordinate \begin{equation}\label{Eq:ProjOperatorLambdaRegularized}
\Lambda_{u,\alpha}=(H^TH)_{\alpha}^{-1}H^T,
\end{equation}
is known, and hence where the estimation is focused on the difference between $H$ and $\hat{H}$. The second term is the term induced by the estimation of the inverse of the Gram matrix $H^TH$. The last term is the one corresponding to the estimated coordinates of $u$ in the regularized basis $H$.
We also consider 
\begin{equation}\label{Eq:ProjOperatorTermTildeIII}
\tilde{\mathrm{III}}_{\alpha}=(u^TH-u^T\hat{H})(H^TH)_{\alpha}^{-1}H^TH(H^TH)_{\alpha}^{-1}(H^Tu-\hat{H}^Tu).
\end{equation}
This term is directly related to $\mathrm{III}_{\alpha}$.
\begin{theorem}\label{Th:ProjOperatorSkeletonProofRegularized}
Let $\delta\in(0,1)$, and $\alpha=2d_{\delta,K}$. We have, for $u\in\RR^{K}$, on the event $\mathcal{B}_{\delta}$,
\begin{equation}
    \|P_{[H]_{\alpha}}(u)-P_{[\hat{H}]_{\alpha}}(u)\|^{2}\le \tilde{C}_{1}\mathrm{I}_{\alpha}+\tilde{C}_{2}\mathrm{Cond}(H^TH)\frac{1}{\rho_{\min}\left((\hat{H}^T\hat{H})_{\alpha}\right)}\mathrm{II}_{\alpha}+\tilde{C}_{3}\mathrm{Cond}(H^TH)^{3}\mathrm{III}_{\alpha}.
\end{equation}
\end{theorem}
\begin{proof}
We use the following decomposition,
\begin{multline*}
P_{[H]_{\alpha}}(u)-P_{[\hat{H}]_{\alpha}}(u)=\\
(H-\hat{H})(\hat{H}^T\hat{H})_{\alpha}^{-1}\hat{H}^{T}u+H\big((H^TH)_{\alpha}^{-1}-(\hat{H}^T\hat{H})_{\alpha}^{-1}\big)\hat{H}^Tu+H(H^TH)_{\alpha}^{-1}(H-\hat{H})^{T}u.
\end{multline*}
It follows that
\begin{align}
    \lefteqn{\|P_{[H]_{\alpha}}(u)-P_{[\hat{H}]_{\alpha}}(u)\|^2}\nonumber\\
    &\le4u^T\hat{H}(\hat{H}^T\hat{H})_{\alpha}^{-1}(\hat{H}-H)^T(\hat{H}-H)(\hat{H}^T\hat{H})_{\alpha}^{-1}\hat{H}^Tu\nonumber\\
    &\quad+4u^T\hat{H}\big((\hat{H}^T\hat{H})_{\alpha}^{-1}-(H^TH)_{\alpha}^{-1}\big)H^TH\big((\hat{H}^T\hat{H})_{\alpha}^{-1}-(H^TH)_{\alpha}^{-1}\big)\hat{H}^Tu\nonumber\\
    &\quad+2(H^Tu-\hat{H}^Tu)^T(H^TH)_{\alpha}^{-1}H^TH(H^TH)_{\alpha}^{-1}(H^Tu-\hat{H}^Tu)\nonumber\\
    &:=4T_{1,\alpha}+4T_{2,\alpha}+2\tilde{\mathrm{III}}_{\alpha}.\label{Eq:ProjOperatorRdecompa}
\end{align}

\subsubsection[First term]{Bound on $T_{1,\alpha}$}
We first consider the term $T_{1,\alpha}$,
$$T_{1,\alpha}=u^T\hat{H}(\hat{H}^T\hat{H})_{\alpha}^{-1}(\hat{H}-H)^T(\hat{H}-H)(\hat{H}^T\hat{H})_{\alpha}^{-1}\hat{H}^Tu.$$
We decompose this term as follows:
\begin{align*}
T_{1,\alpha}&\le2u^T\hat{H}(H^TH)_{\alpha}^{-1}(\hat{H}-H)^T(\hat{H}-H)(H^TH)_{\alpha}^{-1}\hat{H}^Tu\\
&\quad+2u^T\hat{H}((\hat{H}^T\hat{H})_{\alpha}^{-1}-(H^TH)_{\alpha}^{-1})(\hat{H}-H)^T(\hat{H}-H)((\hat{H}^T\hat{H})_{\alpha}^{-1}-(H^TH)_{\alpha}^{-1})\hat{H}^Tu\\
&=:T_{11}+T_{12}.
\end{align*}
First we concentrate our attention on 
\begin{align*}
    T_{12}&\le
    2\ \|E(\hat{H}^T\hat{H})_{\alpha}^{-1}((\hat{H}^T\hat{H})_{\alpha}-(H^TH)_{\alpha})(H^TH)_{\alpha}^{-1}\hat{H}^Tu\|^{2}\\
    &\le2\frac{\rho\big(E^TE\big)}{\rho_{\min}\big((\hat{H}^T\hat{H})_{\alpha}\big)^{2}}u^T\hat{H}(H^TH)_{\alpha}^{-1}((\hat{H}^T\hat{H})_{\alpha}-(H^TH)_{\alpha})^{2}(H^TH)_{\alpha}^{-1}\hat{H}^Tu\\
    &\le 4\frac{\rho\big(E^TE\big)}{\rho_{\min}\big((\hat{H}^T\hat{H})_{\alpha}\big)^{2}}u^TE(H^TH)_{\alpha}^{-1}((\hat{H}^T\hat{H})_{\alpha}-(H^TH)_{\alpha})^{2}(H^TH)_{\alpha}^{-1}E^Tu\\
    &\quad+4\frac{\rho\big(E^TE\big)}{\rho_{\min}\big((\hat{H}^T\hat{H})_{\alpha}\big)^{2}}u^TH(H^TH)_{\alpha}^{-1}((\hat{H}^T\hat{H})_{\alpha}-(H^TH)_{\alpha})^{2}(H^TH)_{\alpha}^{-1}H^Tu.
\end{align*}
Thanks to \cref{Prop:EigenControl} we get 
\begin{equation}\label{Eq:ProjOperatorRLema}
\frac{\rho\big(E^TE\big)}{\rho_{\min}\big((\hat{H}^T\hat{H})_{\alpha}\big)}\le \frac{\alpha}{2}\frac{1}{\alpha}\le \frac{1}{2}.
\end{equation}
\begin{align}
    \lefteqn{T_{12}\le}\nonumber\\
    &\ 2\frac{1}{\rho_{\min}\left((\hat{H}^T\hat{H})_{\alpha}\right)}u^TE(H^TH)_{\alpha}^{-1}((\hat{H}^T\hat{H})_{\alpha}-(H^TH)_{\alpha})^{2}(H^TH)_{\alpha}^{-1}E^Tu\nonumber\\
    &\quad+2\frac{1}{\rho_{\min}\left((\hat{H}^T\hat{H})_{\alpha}\right)}u^TH(H^TH)_{\alpha}^{-1}((\hat{H}^T\hat{H})_{\alpha}-(H^TH)_{\alpha})^{2}(H^TH)_{\alpha}^{-1}H^Tu\nonumber\\
    &\le2 \frac{\rho\big((\hat{H}^T\hat{H})_{\alpha}-(H^TH)_{\alpha}\big)^2}{\rho_{\min}\big((\hat{H}^T\hat{H})_{\alpha}\big)\rho_{\min}\big((H^TH)_{\alpha}\big)}(u^T\hat{H}-u^TH)(H^TH)_{\alpha}^{-1}(H^TH)_{\alpha}(H^TH)_{\alpha}^{-1}(\hat{H}^Tu-H^Tu)\nonumber\\
    &\quad+2u^TH(H^TH)_{\alpha}^{-1}((\hat{H}^T\hat{H})_{\alpha}-(H^TH)_{\alpha})(\hat{H}^T\hat{H})_{\alpha}^{-1}((\hat{H}^T\hat{H})_{\alpha}-(H^TH)_{\alpha})(H^TH)_{\alpha}^{-1}H^Tu
\end{align}
Thanks to \cref{Prop:EigenControl} we have
\begin{align*}
\frac{\rho(\hat{H}^T\hat{H}-H^TH)^{2}}{\rho_{\min}\big((\hat{H}^T\hat{H})_{\alpha}\big)\rho_{\min}\big((H^TH)_{\alpha}\big)}&\le \frac{2\alpha^{2}+2\rho(H^TH)^{2}}{\rho_{\min}\big((\hat{H}^T\hat{H})_{\alpha}\big)\rho_{\min}\big((H^TH)_{\alpha}\big)}\\
&\le 2+4\mathrm{Cond}(H^TH)\\
&\le 8\mathrm{Cond}(H^TH).
\end{align*}
\begin{align}\label{Eq:ProjOperatorRThreeStars}
T_{12}&\le 2\cdot 32 \mathrm{Cond}(H^TH)^{2}\mathrm{III}_{\alpha}+2\frac{1}{\rho_{\min}\left((\hat{H}^T\hat{H})_{\alpha}\right)}\mathrm{II}_{\alpha}.
\end{align}
Now we propose a bound for $T_{11}$.
\begin{align}
    T_{11}&\le
    \ 2u^T\hat{H}(H^TH)_{\alpha}^{-1}(\hat{H}-H)^T(\hat{H}-H)(H^TH)_{\alpha}^{-1}\hat{H}^Tu\nonumber\\
    &\le 4u^TH(H^TH)_{\alpha}^{-1}(\hat{H}-H)^{T}(\hat{H}-H)(H^TH)_{\alpha}^{-1}H^Tu\nonumber\\
    &\quad+4(u^T\hat{H}-u^TH)(H^TH)_{\alpha}^{-1}(\hat{H}-H)^T(\hat{H}-H)(H^TH)_{\alpha}^{-1}(\hat{H}^Tu-H^Tu)\nonumber\\
    &\le 4\mathrm{I}_{\alpha}+4\frac{\rho\big((\hat{H}-H)^T(\hat{H}-H)\big)}{\rho_{\min}\big((H^TH)_{\alpha}\big)}\mathrm{III}_{\alpha}\nonumber\\
    &\le 4\mathrm{I}_{\alpha} + 2\mathrm{III}_{\alpha}.\label{Eq:ProjOperatorRFourstars}
\end{align}
The last step is obtained by \eqref{Eq:ProjOperatorRLema}.
\paragraph{Final bound on $T_{1,\alpha}$}
We deduce from \eqref{Eq:ProjOperatorRThreeStars} and \eqref{Eq:ProjOperatorRFourstars} that
\begin{equation}\label{Eq:ProjOperatorT1AlphR}
    T_{1,\alpha}\le 4\mathrm{I}_{\alpha} + \big(2+64\mathrm{Cond}(H^TH)\big)\mathrm{III}_{\alpha}+ 2\frac{1}{\rho_{\min}\left((\hat{H}^T\hat{H})_{\alpha}\right)}\mathrm{II}_{\alpha}.
\end{equation}
\subsubsection[Second bound]{Bound on $T_{2,\alpha}$}
We focus on 
$$T_{2,\alpha}=u^T\hat{H}\big((\hat{H}^T\hat{H})_{\alpha}^{-1}-(H^TH)_{\alpha}^{-1}\big)H^TH\big((\hat{H}^T\hat{H})_{\alpha}^{-1}-(H^TH)_{\alpha}^{-1}\big)\hat{H}^Tu.$$
We recall that $\Lambda_{u,\alpha}=(H^TH)_{\alpha}^{-1}H^Tu$. We use the following decomposition,
\begin{align*}
    T_{2,\alpha}&\le \ 2u^TH\big((\hat{H}^T\hat{H})_{\alpha}^{-1}-(H^TH)_{\alpha}^{-1}\big)H^TH\big((\hat{H}^T\hat{H})_{\alpha}^{-1}-(H^TH)_{\alpha}^{-1}\big)H^Tu\\
    &  \quad+ 2u^TE\big((\hat{H}^T\hat{H})_{\alpha}^{-1}-(H^TH)_{\alpha}^{-1}\big)H^TH\big((\hat{H}^T\hat{H})_{\alpha}^{-1}-(H^TH)_{\alpha}^{-1}\big)E^Tu\\
    &\le 2\Lambda_{u,\alpha}^T\big((\hat{H}^T\hat{H})_{\alpha}-(H^TH)_{\alpha}\big)(\hat{H}^T\hat{H})_{\alpha}^{-1}H^TH(\hat{H}^T\hat{H})_{\alpha}^{-1}\big((\hat{H}^T\hat{H})_{\alpha}-(H^TH)_{\alpha}\big)\Lambda_{u,\alpha}\\
    & \quad+ 2 \|H(\hat{H}^T\hat{H})_{\alpha}^{-1}((\hat{H}^T\hat{H})_{\alpha}-(H^TH)_{\alpha})(H^TH)_{\alpha}^{-1}E^Tu\|^{2}\\
    &\le 2\frac{\rho(H^TH)}{\rho_{\min}\big((\hat{H}^T\hat{H})_{\alpha}\big)^2}u^TH(H^TH)_{\alpha}^{-1}((\hat{H}^T\hat{H})_{\alpha}-(H^TH)_{\alpha})^2(H^TH)_{\alpha}^{-1}H^Tu\\
    & \quad+ 2 \frac{\rho(H^TH)}{\rho_{\min}\big((\hat{H}^T\hat{H})_{\alpha}\big)^{2}}\frac{\rho\big(\hat{H}^T\hat{H}-H^TH\big)^2}{\rho_{\min}\big((H^TH)_{\alpha}\big)}\mathrm{III}_{\alpha}.
\end{align*}
Thanks to \cref{Prop:EigenControl}, we have
\begin{align*}
\frac{\rho(\hat{H}^T\hat{H}-H^TH)^{2}}{\rho_{\min}\big((\hat{H}^T\hat{H})_{\alpha}\big)\rho_{\min}\big((H^TH)_{\alpha}\big)}&\le \frac{2\alpha^{2}+2\rho(H^TH)}{\rho_{\min}\big((\hat{H}^T\hat{H})_{\alpha}\big)\rho_{\min}\big((H^TH)_{\alpha}\big)}\\
&\le 2+4\mathrm{Cond}(H^TH)\\
&\le 8\mathrm{Cond}(H^TH).
\end{align*}
Thus, 
\begin{align}
    T_{2,\alpha}&\le 4\frac{\rho(H^TH)}{\rho_{\min}(H^TH)}\frac{1}{\rho_{\min}\left((\hat{H}^T\hat{H})_{\alpha}\right)}\mathrm{II}_{\alpha}+32\mathrm{Cond}(H^TH)^{3}\mathrm{III}_{\alpha}.
\end{align}
\paragraph{Final bound on $T_{2}$}
We deduce from what precedes that
\begin{equation}\label{Eq:ProjOperatorT2AlphR}
    T_{2,\alpha}\le4\mathrm{Cond}(H^TH)\frac{1}{\rho_{\min}\left((\hat{H}^T\hat{H})_{\alpha}\right)}\mathrm{II}_{\alpha}+32\mathrm{Cond}(H^TH)^{3}\mathrm{III}_{\alpha}.
\end{equation}
\paragraph*{Bound on $\tilde{\mathrm{III}}_{\alpha}$}
We recall the expression of the term $\tilde{\mathrm{III}}_{\alpha}$ given in \eqref{Eq:ProjOperatorTermTildeIII}: $$\tilde{\mathrm{III}}_{\alpha}=(u^TH-u^T\hat{H})(H^TH)_{\alpha}^{-1}H^TH(H^TH)_{\alpha}^{-1}(H^Tu-\hat{H}^Tu).$$
The fact that $\rho\bigg(H^TH(H^TH)_{\alpha}^{-1}\bigg)\le 1$, directly implies 
\begin{equation}\label{Eq:ProjOperatorTildeR}
\tilde{\mathrm{III}}_{\alpha}\le \mathrm{III}_{\alpha}.
\end{equation}
\paragraph{Final Bound}
The decomposition \eqref{Eq:ProjOperatorRdecompa}, and inequalities \eqref{Eq:ProjOperatorT1AlphR}, \eqref{Eq:ProjOperatorT2AlphR}  and \eqref{Eq:ProjOperatorTildeR} yield
\begin{multline}
\|(P_{[\hat{H}]_{\alpha}}-P_{[H]_{\alpha}})(u)\|^{2}\le16\mathrm{I}_{\alpha}+ (8+16\cdot4\mathrm{Cond}(H^TH))\frac{1}{\rho_{\min}\left((\hat{H}^T\hat{H})_{\alpha}\right)}\mathrm{II}_{\alpha}\\ 
+(9+4\cdot64\mathrm{Cond}(H^TH)+4\cdot32\mathrm{Cond}(H^TH)^{3})\mathrm{III}_{\alpha}.
\end{multline}
This concludes the proof.
\end{proof}
\subsection{Detailed proofs for each Scenario (Ridge version)}\label{Sec:ProjOperatorDetailed}
The final step is to upper bound the terms $\mathrm{I}_{\alpha}, \mathrm{II}_{\alpha}$ and $\mathrm{III}_{\alpha}$ introduced in \eqref{Eq:ProjOperatorTermIAlph}, \eqref{Eq:ProjOperatorTermIIAlph}, and \eqref{Eq:ProjOperatorTermIIIAlph} for each scenario  by taking the supremum on all $u$ such that $u^Tu=1$.  We have fixed $\alpha$ thanks to \cref{Prop:EigenControl}.  We propose a bound for the three terms in the regularized cases, for each scenario. 
We recall that $\Lambda_{u}$ and $\Lambda_{u,\alpha}$ are defined in \eqref{Eq:ProjOperatorLambda} and \eqref{Eq:ProjOperatorLambdaRegularized}.
\paragraph*{}All along this proof, we will use \cref{Prop:DeviationResult} several times (this number is independent of $n$ and $K$), for each use we will replace $\delta$ by $\frac{\delta}{2}$ in the event $\mathcal{B}_{\delta}$ in \cref{Def:EventB} in order to create a new event which will be the intersection of $\mathcal{B}_{\delta}$ and the one created by \cref{Prop:DeviationResult}. It allows to simultaneously have the property of \cref{Prop:DeviationResult} and the inner property of $\mathcal{B}_{\delta}$ on $\mathrm{Tr}(E^TE)$ with a probability higher than $1-\delta$. 
\paragraph*{The case of term $\mathrm{II}_{\alpha}$} The term $\mathrm{II}$ is more complex than terms $\mathrm{I}$ and $\mathrm{III}$ and will be handled as follows. First, recall that $\Lambda_{u}=(H^TH)^{-1}H^Tu$. We have
\begin{align*}
    \|(\hat{H}^T\hat{H}-H^TH)\Lambda_{u,\alpha}\|^2&=\|(E^TH+H^TE+E^TE)\Lambda_{u,\alpha}\|^2\\
    &\le2\|(E^TH+H^TE)\Lambda_{u,\alpha}\|^2+2\|E^TE\Lambda_{u,\alpha}\|^2\\
    &\le 4\|E^TH\Lambda_{u,\alpha}\|^2+4\|H^TE\Lambda_{u,\alpha}\|^2+2\|E^TE\Lambda_{u,\alpha}\|^2.
\end{align*}
Hence, 
\begin{equation}\label{Eq:ProjOperatorDecompIIAlph}
\mathrm{II}_{\alpha}\le 4\|E^TH\Lambda_{u,\alpha}\|^2+4\|H^TE\Lambda_{u,\alpha}\|^2+2\|E^TE\Lambda_{u,\alpha}\|^2.
\end{equation}
The term $\mathrm{II}_{\alpha}$ will be handled after these three terms have been bounded for each Scenario. 
\subsubsection{Scenario 1 and 2}\label{Subsub:Scen12Reg}
We recall that the terms $\mathrm{I}_{\alpha}, \mathrm{II}_{\alpha}$ and $\mathrm{III}_{\alpha}$ are introduced in \eqref{Eq:ProjOperatorTermIAlph}, \eqref{Eq:ProjOperatorTermIIAlph}, and \eqref{Eq:ProjOperatorTermIIIAlph}. The terms $\Lambda_{u}$ and $\Lambda_{u,\alpha}$ are defined in \eqref{Eq:ProjOperatorLambda} and \eqref{Eq:ProjOperatorLambdaRegularized}.
\subsubsection*{Scenario 1 and 2: Term $\mathrm{I}_{\alpha}$}
We use the same method as case without the regularization. We recall that 
$$\mathrm{I}_{\alpha}=\|E(H^TH)_{\alpha}^{-1}H^Tu\|^{2}.$$ 
Using the deviation result of \cref{Prop:DeviationResult}, we get, on $\mathcal{B}_{\delta}$, 
$$\mathrm{I}_{\alpha}\le C_{\delta}\gamma^{2}n\Lambda_{u,\alpha}^TS\Lambda_{u,\alpha}\le C_{\delta,}n\gamma^{2}u^TH(H^TH)_{\alpha}^{-1}S(H^TH)_{\alpha}^{-1}H^Tu.$$
We use the following to obtain the same bounds as for the non regularized version. We take the $\sup$ on all $u\in\RR^{n}$ such that $u^Tu=1$,
\begin{align*}
    \mathrm{I}_{\alpha}
    &\le C_{\delta}n\gamma^{2}\rho\bigg(H(H^TH)_{\alpha}^{-1}S(H^TH)_{\alpha}^{-1}H^T\bigg)\\
    &\le C_{\delta}n\gamma^{2}\rho\bigg((H^TH)(H^TH)_{\alpha}^{-1}\bigg)\rho\bigg(S(H^TH)_{\alpha}^{-1}\bigg)\\
    &\le C_{\delta}n\gamma^{2}\rho\bigg(S(H^TH)^{-1}\bigg).
\end{align*}
Where the last step results from 
\begin{align*}
\rho\bigg(S(H^TH)_{\alpha}^{-1}\bigg)&=\rho\bigg(S(H^TH)^{-1}(H^TH)(H^TH)_{\alpha}^{-1}\bigg)\\
&\le \rho\bigg(S(H^TH)^{-1}\bigg)\rho\bigg(H^TH(H^TH)_{\alpha}^{-1}\bigg)\\
&\le \rho\bigg(S(H^TH)^{-1}\bigg).
\end{align*}
\paragraph*{Final bound on $\mathrm{I}_{\alpha}$ :}
On the event $\mathcal{B}_{\delta}$, we have
\begin{equation}\label{Eq:ProjOperatorS12TermIAlph}
\mathrm{I}_{\alpha}\le C_{\delta}n\gamma^{2}\rho\left(S(H^TH)^{-1}\right).
\end{equation}
\subsubsection*{Scenario 1 and 2: Term $\mathrm{II}_{\alpha}$}
We recall that the second term is induced by the estimation of the inverse of the Gram matrix $H^TH$. It is defined by 
$$u^TH(H^TH)_{\alpha}^{-1}\big((\hat{H}^T\hat{H})_{\alpha}-(H^TH)_{\alpha}\big)^{2}(H^TH)_{\alpha}^{-1}H^Tu.$$
We will use the decomposition of the Term $\mathrm{II}_{\alpha}$ given in \eqref{Eq:ProjOperatorDecompIIAlph}. It yields three terms $\|H^TE\Lambda_{u,\alpha}\|^{2},$ $\|E^TH\Lambda_{u,\alpha}\|^{2}$ and $\|E^TE\Lambda_{u,\alpha}\|^{2}$ that we will study below. The vectors $E^TH\Lambda_{u,\alpha}$ and $H^TE\Lambda_{u,\alpha}$ follow centered multivariate normal distributions. First, $$\mathrm{Var}(H^TE\Lambda_{u,\alpha})=H^T\mathrm{Var}(E\Lambda_{u,\alpha})H=\gamma^{2}\Lambda_{u,\alpha}^TS\Lambda_{u,\alpha}H^TH.$$ We have $(E^TH\Lambda_{u,\alpha})_i=E_{\cdot,i}^TH\Lambda_{u,\alpha}$, We can then compute 
$$\EE[E_{\cdot,i}^TH\Lambda_{u,\alpha}E_{\cdot,j}^TH\Lambda_{u,\alpha}]=\gamma^{2}\Lambda_{u,\alpha}^TH^TH\Lambda_{u,\alpha}S_{ij}.$$
Hence we get
$\mathrm{Var}(E^TH\Lambda_{u,\alpha})=\gamma^{2}\|H\Lambda_{u,\alpha}\|^2S$. On $\mathcal{B}_{\delta}$, thanks to the deviation results of \cref{Prop:DeviationResult}, we get
\begin{align*}
\|H^TE\Lambda_{u,\alpha}\|^{2}&\le C_{\delta}\gamma^{2}\Lambda_{u,\alpha}^TS\Lambda_{u,\alpha}\mathrm{Tr}(H^TH),\\
\|E^TH\Lambda_{u,\alpha}\|^{2}&\le C_{\delta}\gamma^{2}\|H\Lambda_{u,\alpha}\|^{2}\mathrm{Tr}(S).
\end{align*}
Using the fact that $\rho(E^TE)\le \frac{\alpha}{2}$ we also obtain 
$$\|E^TE\Lambda_{u,\alpha}\|^{2}\le \rho(E^TE)\|E\Lambda_{u,\alpha}\|^{2}\le \frac{\alpha}{2}\|E\Lambda_{u,\alpha}\|^{2}.$$
Applying the deviation result of \cref{Prop:DeviationResult} on these norms of Gaussian vectors, we get, on the event $\mathcal{B}_{\delta}$,
$$\mathrm{II}_{\alpha}\le C_{\delta}\gamma^{2}\big(\Lambda_{u,\alpha}^TS\Lambda_{u,\alpha}\mathrm{Tr}(H^TH)+\|H\Lambda_{u,\alpha}\|^{2}\mathrm{Tr}(S)+\frac{1}{2} \alpha n\Lambda_{u,\alpha}^TS\Lambda_{u,\alpha}\big).$$
As for the proof of the non regularized estimation, we consider here $\rho_{\min}\big((\hat{H}^T\hat{H})_{\alpha}\big)^{-1}\mathrm{II}_{\alpha}$.
We study separately the terms appearing in the right-hand side above. First,
\begin{align*}
    \Lambda_{u,\alpha}^TS\Lambda_{u,\alpha}&\le\rho\bigg(H(H^TH)_{\alpha}^{-1}S(H^TH)_{\alpha}^{-1}H^T\bigg)\\
    &\le \rho\bigg(S^{\frac{1}{2}}(H^TH)_{\alpha}^{-1}H^TH(H^TH)_{\alpha}^{-1}S^{\frac{1}{2}}\bigg)\\
    &\le \rho\bigg(S (H^TH)_{\alpha}^{-1}H^TH(H^TH)_{\alpha}^{-1}\bigg)\\
    &\le \rho\bigg(S (H^TH)_{\alpha}^{-1}\bigg)\rho\bigg(H^TH(H^TH)_{\alpha}^{-1}\bigg)\\
    &\le \rho\bigg(S(H^TH)^{-1}\bigg).
\end{align*}
In the same way, we get 
\begin{align*}
    \|H\Lambda_{u,\alpha}\|^{2}&\le \rho\bigg(H(H^TH)_{\alpha}^{-1}H^TH(H^TH)_{\alpha}^{-1}H^T\bigg)\\
    &\le\rho\bigg((H^TH)_{\alpha}^{-\frac{1}{2}}H^TH(H^TH)_{\alpha}^{-\frac{1}{2}}\bigg)\rho\bigg(H(H^TH)_{\alpha}^{-1}H^T\bigg)\\
    &\le \rho\bigg((H^TH)_{\alpha}^{-1}H^TH\bigg)^{2}\\
    &\le 1.
\end{align*}
Hence we deduce
\begin{align*}
    \frac{\mathrm{II}_{\alpha}}{\rho_{\min}\big((\hat{H}^T\hat{H})_{\alpha}\big)}&\le C_{\delta}\gamma^{2}\left(\frac{2\rho\big(S (H^TH)^{-1}\big)\mathrm{Tr}(H^TH)}{\rho_{\min}(H^TH)}+\frac{2\mathrm{Tr}(S)}{\rho_{\min}(H^TH)}+\frac{1}{2}\rho\big(S(H^TH)^{-1}\big)\right)\\
    &\le C_{\delta}\gamma^{2}\left(2K\mathrm{Cond}(H^TH)\rho\big(S(H^TH)^{-1}\big)+\frac{2\mathrm{Tr}(S)}{\rho_{\min}(H^TH)}+\frac{1}{2}\rho\big(S(H^TH)^{-1}\big)\right),
\end{align*}
where we used the inequality $\rho_{\min}\left(\hat{H}^T\hat{H})_{\alpha}\right)\ge \max(\alpha,\frac{\rho_{\min}(H^TH)}{2})$ which follows from \cref{Prop:EigenControl}.
\paragraph*{Final bound on $\rho_{\min}\left((\hat{H}^T\hat{H})_{\alpha}\right)^{-1}\mathrm{II}_{\alpha}$ :}
\begin{equation}\label{Eq:ProjOperatorS12TermIIAlph}
\frac{\mathrm{II}_{\alpha}}{\rho_{\min}((\hat{H}^T\hat{H})_{\alpha})}\le C_{\delta}\gamma^{2}\left(2K\mathrm{Cond}(H^TH)\rho\big(S(H^TH)^{-1}\big)+\frac{2\mathrm{Tr}(S)}{\rho_{\min}(H^TH)}+\frac{1}{2}\rho\big(S(H^TH)^{-1}\big)\right).
\end{equation}
\subsubsection*{Scenario 1 and 2: Term $\mathrm{III}_{\alpha}$}
The last term to study is the one corresponding to the estimated coordinates of $u$ in the regularized basis $H$.
It is defined by 
$$\mathrm{III}_{\alpha}=(u^TH-u^T\hat{H})(H^TH)_{\alpha}^{-1}(H^Tu-\hat{H}^Tu).$$
We proceed as for the term $\mathrm{III}$, by applying the deviation result of \cref{Prop:DeviationResult}. It results that, on $\mathcal{B}_{\delta}$,
\begin{equation}\label{Eq:ProjOperatorIIIAlphScen12}
\mathrm{III}_{\alpha}\le C_{\delta}\|u\|^{2} \gamma^{2}\mathrm{Tr}\big((H^TH)_{\alpha}^{-1}S\big).
\end{equation}
It remains to consider $\mathrm{Tr}\big((H^TH)_{\alpha}^{-1}S\big)$. We can bound this quantity by
\begin{align*}
    \mathrm{Tr}\big((H^T
H)_{\alpha}^{-1}S\big)&\le\mathrm{Tr}\bigg((H^TH)_{\alpha}^{-1}H^TH(H^TH)^{-1}S\bigg)\\
    &\le \mathrm{Tr}\bigg((H^TH)_{\alpha}^{-1}H^TH(H^TH)^{-\frac{1}{2}}S(H^TH)^{-\frac{1}{2}}\bigg)
\end{align*}
We use the fact that $(H^TH)^{-\frac{1}{2}}$ is a polynomial in $H^TH$ and then commute with $H^TH$ and $(H^TH)_{\alpha}^{-1}$.
\begin{align}    
    \mathrm{Tr}\big((H^TH)_{\alpha}^{-1}S\big)&\le \rho\bigg((H^TH)_{\alpha}^{-1}H^TH\bigg)\mathrm{Tr}\bigg((H^TH)^{-\frac{1}{2}}S(H^TH)^{-\frac{1}{2}}\bigg)\nonumber\\
    &\le \mathrm{Tr}\bigg((H^TH)^{-1}S\bigg)\label{Eq:ProjOperatorAlphTr},
\end{align}
We used in the last step the inequality \eqref{Eq:ProjOperatorTraceTrick} with the positive symmetric matrices\\ $(H^TH)^{\frac{1}{2}}(H^TH)_{\alpha}^{-1}(H^TH)^{\frac{1}{2}}$ and $(H^TH)^{-\frac{1}{2}}S(H^TH)^{-\frac{1}{2}}$.
Inequalities \eqref{Eq:ProjOperatorIIIAlphScen12} and \eqref{Eq:ProjOperatorAlphTr} lead to, on $\mathcal{B}_{\delta}$,
$$\mathrm{III}_{\alpha}\le C_{\delta}\gamma^{2}\mathrm{Tr}\big((H^TH)^{-1}S\big).$$
\paragraph*{Final bound for $\mathrm{III}_{\alpha}$ :}
\begin{equation}\label{Eq:ProjOperatorS12TermIIIAlph}
C_{\delta}\gamma^{2}\mathrm{Tr}\big((H^TH)^{-1}S\big).
\end{equation}
This concludes the proofs of theorems \cref{Th:ProjOperatorRowsPenalinise} and \cref{Th:ProjOperatorIndependentPenalinise} by setting $S=I_{K}$ thanks to \cref{Th:ProjOperatorSkeletonProofRegularized} by regrouping the three different main terms \eqref{Eq:ProjOperatorS12TermIAlph}, \eqref{Eq:ProjOperatorS12TermIIAlph}, and \eqref{Eq:ProjOperatorS12TermIIIAlph} with the bias term of \cref{Subsec:ProjOperatorBiasRegularized}. 
\subsubsection{Scenario 3}
We recall that the terms $\mathrm{I}_{\alpha}, \mathrm{II}_{\alpha}$ and $\mathrm{III}_{\alpha}$ are introduced in \eqref{Eq:ProjOperatorTermIAlph}, \eqref{Eq:ProjOperatorTermIIAlph}, and \eqref{Eq:ProjOperatorTermIIIAlph}. The terms $\Lambda_{u}$ and $\Lambda_{u,\alpha}$ are defined in \eqref{Eq:ProjOperatorLambda} and \eqref{Eq:ProjOperatorLambdaRegularized}.
\subsubsection*{Scenario 3: Term $\mathrm{I}_{\alpha}$}
As for the regularized independent rows case in Scenario 2, we use the previous bounds obtained for the independent columns case. 
We get on the event $\mathcal{B}_{\delta},$
$$\mathrm{I}_{\alpha}\le C_{\delta} \gamma^{2}\mathrm{Tr}(A)\|\Lambda_{u,\alpha}\|^{2}.$$
Let us bound $\|\Lambda_{u,\alpha}\|^{2}$. We have, for $u\in\mathbb{R}^{n}$ such that $u^Tu=1$,
\begin{align*}
    \|\Lambda_{u,\alpha}\|^{2}&=\|(H^TH)_{\alpha}^{-1}H^Tu\|^{2}\\
    &=u^TH(H^TH)_{\alpha}^{-1}(H^TH)_{\alpha}^{-1}H^Tu\\
    &\le \rho\left(H(H^TH)_{\alpha}^{-1}(H(H^TH)_{\alpha}^{-1})^{T}\right)\\
    &\le \rho\left((H^TH)_{\alpha}^{-1}H^TH(H^TH)_{\alpha}^{-1}\right)\\
    &\le\rho\left((H^TH)_{\alpha}^{-1}H^TH\right)\rho_{\min}(H^TH)^{-1}\\
    &\le \rho_{\min}(H^TH)^{-1}.
\end{align*}
\paragraph*{Final bound on $\mathrm{I}_{\alpha}$ in Scenario 3}
Finally, we obtain 
\begin{equation}\label{Eq:ProjOperatorS3TermIAlph}
\mathrm{I}_{\alpha}\le C_{\delta}\gamma^{2}\frac{\mathrm{Tr}(A)}{\rho_{\min}(H^TH)}.
\end{equation}
\subsubsection*{Scenario 3: Term $\mathrm{II}_{\alpha}$}
We focus first on $\rho_{\min}\big((\hat{H}^T\hat{H})_{\alpha}^{-1}\big)\mathrm{II}_{\alpha}.$
We decompose $\mathrm{II}_{\alpha}$ in three terms. We recall the inequality \eqref{Eq:ProjOperatorDecompIIAlph}:
\begin{equation}\label{Eq:ProjOperatoratriple}
\mathrm{II}_{\alpha} \le 4\|E^TH\Lambda_{u,\alpha}\|^2+4\|H^TE\Lambda_{u,\alpha}\|^2+2\|E^TE\Lambda_{u,\alpha}\|^2.
\end{equation}
We focus on terms $\|E^TH\Lambda_{u,\alpha}\|^{2}$ and $\|H^TE\Lambda_{u,\alpha}\|^{2}$. For $H^TE\Lambda_{u,\alpha}$, we have 
$$\mathrm{Var}(H^TE\Lambda_{u,\alpha})=H^T\mathrm{Var}(E\Lambda_{u,\alpha})H=\gamma^{2}H^TAH\|\Lambda_{u,\alpha}\|^{2}.$$
Using the deviation result of \cref{Prop:DeviationResult}, we deduce
\begin{equation}\label{Eq:ProjOperatorbtriple}
\|H^TE\Lambda_{u,\alpha}\|^{2}\le C_{\delta}\gamma^{2}\|\Lambda_{u,\alpha}\|^{2}\mathrm{Tr}(H^TAH).
\end{equation}
We have $(E^TH\Lambda_{u,\alpha})_{i}=E_{\cdot,i}^TH\Lambda_{u,\alpha}$. Hence, 
$$\EE[\Lambda_{u,\alpha}^TH^T E_{\cdot,i}E_{\cdot,j}^{T}H\Lambda_{u,\alpha}]=\gamma^{2}\delta_{ij}\Lambda_{u,\alpha}^TH^TAH\Lambda_{u,\alpha}.$$
Therefore $\mathrm{Var}(E^TH\Lambda_{u,\alpha})=\gamma^{2}\Lambda_{u,\alpha}^TH^TAH\Lambda_{u,\alpha}I_{K}.$
Then, on $\mathcal{B}_{\delta}$,
\begin{equation}\label{Eq:ProjOperatorctriple}
\|E^TH\Lambda_{u,\alpha}\|^{2}\le C_{\delta}\gamma^{2}K\Lambda_{u,\alpha}^{T}H^TAH\Lambda_{u,\alpha}.
\end{equation}
For the last term, we use the following inequalities,
\begin{equation}\label{Eq:ProjOperatordtriple}
\|E^TE\Lambda_{u,\alpha}\|^{2}\le \rho(E^TE)\|E\Lambda_{u,\alpha}\|^{2}\le \frac{\alpha}{2}\|E\Lambda_{u,\alpha}\|^{2}.
\end{equation}
\eqref{Eq:ProjOperatoratriple}, \eqref{Eq:ProjOperatorbtriple}, \eqref{Eq:ProjOperatorctriple} and \eqref{Eq:ProjOperatordtriple} give
$$\mathrm{II}_{\alpha}\le C_{\delta}\gamma^{2}\bigg(K\Lambda_{u,\alpha}^TH^TAH\Lambda_{u,\alpha}+\|\Lambda_{u,\alpha}\|^{2}\mathrm{Tr}(H^TAH)+\frac{1}{2}\alpha\|\Lambda_{u,\alpha}\|^{2}\mathrm{Tr}(A)\bigg).$$
We need to bound the terms $\Lambda_{u,\alpha}H^TAH\Lambda_{u,\alpha}$ and $\|\Lambda_{u,\alpha}\|^{2}$ for $u\in\mathbb{R}^{n}$ such that $u^Tu=1$. 
We already have $\|\Lambda_{u,\alpha}\|^{2}\le 1.$
Then 
\begin{align*}
    \Lambda_{u,\alpha}^TH^TAH\Lambda_{u,\alpha}&\le \rho\bigg(H(H^TH)_{\alpha}^{-1}H^TAH(H^TH)_{\alpha}^{-1}H^T\bigg)\\
    &\le\rho\bigg(A^{\frac{1}{2}}\big(H(H^TH)_{\alpha}^{-1}H^T\big)^{2}A^{\frac{1}{2}}\bigg)\\
    &\le \rho\bigg(A\big(H(H^TH)_{\alpha}^{-1}H^T\big)^{2}\bigg)\\
    &\le \rho(A)\rho\bigg(H(H^TH)_{\alpha}^{-1}H^T\bigg)^{2}\\
    &\le \rho(A).
\end{align*}
Using additionally that $\rho_{\min}\left((H^TH)_{\alpha}\right)\ge\max(\alpha,\frac{\rho_{\min}(H^TH)}{2})$ thanks to \cref{Prop:EigenControl}, we finally get 
$$\rho_{\min}\big((\hat{H}^T\hat{H})_{\alpha}\big)^{-1}
\mathrm{II}_{\alpha}\le C_{\delta}\gamma^{2}\bigg(2K\frac{\rho(A)}{\rho_{\min}(H^TH)}+2\frac{\mathrm{Tr}(H^TAH)}{\rho_{\min}(H^TH)}+\frac{1}{2}\frac{\mathrm{Tr}(A)}{\rho_{\min}(H^TH)}\bigg).$$
Using the inequality $\mathrm{Tr}(H^TAH)\le K\rho(A)\rho(H^TH)$ we simplify this bound.
\paragraph*{Final bound on $\rho_{\min}\left((\hat{H}^T\hat{H})_{\alpha}\right)^{-1}\mathrm{II}_{\alpha}$}
\begin{equation}\label{Eq:ProjOperatorS3TermIIAlph}
C_{\delta}\gamma^{2}\bigg(2K\frac{\rho(A)}{\rho_{\min}(H^TH)}+2\frac{K\rho(A)\rho(H^TH)}{\rho_{\min}(H^TH)}+\frac{1}{2}\frac{\mathrm{Tr}(A)}{\rho_{\min}(H^TH)}\bigg).
\end{equation}
\subsubsection*{Scenario 3: Term $\mathrm{III}_{\alpha}$}
Doing the same for $\mathrm{III}_{\alpha}$, we consider the Gaussian vector $(H^TH)_{\alpha}^{-\frac{1}{2}}E^Tu$
The associated variance matrix is $$(H^TH)_{\alpha}^{-1/2}\mathrm{Var}(E^Tu)(H^TH)_{\alpha}^{-1/2}=(H^TH)_{\alpha}^{-1/2}\gamma^{2}u^TAu(H^TH)_{\alpha}^{-1/2}.$$
Thanks to the deviation result of \cref{Prop:DeviationResult}, we get, on $\mathcal{B}_{\delta},$
$$\mathrm{III}_{\alpha}\le C_{\delta}u^TAu\ \mathrm{Tr}\big((H^TH)_{\alpha}^{-1}\big).$$
We clearly have by taking the supremum on $u\in\mathbb{R}^{n}$ such that $u^Tu=1$, it is straightforward that, on $\mathcal{B}_{\delta}$,
\begin{equation}\label{Eq:ProjOperatorS3TermIIIAlph}
\mathrm{III}_{\alpha}\le C_{\delta}\gamma^{2}\rho(A)\mathrm{Tr}\big((H^TH)^{-1}\big).
\end{equation}
We conclude the proof of \cref{Th:ProjOperatorColumnsPenalinise} thanks to the decomposition of \cref{Th:ProjOperatorSkeletonProofRegularized} and the inequalities \eqref{Eq:ProjOperatorS3TermIAlph}, \eqref{Eq:ProjOperatorS3TermIIAlph} and \eqref{Eq:ProjOperatorS3TermIIIAlph}.
\subsubsection{Scenario 4}\label{Subsub:Scen4Reg}
We recall that the terms $\mathrm{I}_{\alpha}, \mathrm{II}_{\alpha}$ and $\mathrm{III}_{\alpha}$ are introduced in \eqref{Eq:ProjOperatorTermIAlph}, \eqref{Eq:ProjOperatorTermIIAlph}, and \eqref{Eq:ProjOperatorTermIIIAlph}. The terms $\Lambda_{u}$ and $\Lambda_{u,\alpha}$ are defined in \eqref{Eq:ProjOperatorLambda} and \eqref{Eq:ProjOperatorLambdaRegularized}.
\subsubsection*{Scenario 4: Term $\mathrm{I}_{\alpha}$}
We consider the distribution of $E\Lambda_{u,\alpha}$. We have 
$$\mathrm{Var}(E\Lambda_{u,\alpha})=\gamma^{2}\sum_{l,m=1}^{K}(\Lambda_{u,\alpha})_{l}(\Lambda_{u,\alpha})_{m}A_{l}VA_{m}^{T}.$$
Thanks to the deviation result of \cref{Prop:DeviationResult}, we have, on the event $\mathcal{B}_{\delta}$,
\begin{align*}
\|E\Lambda_{u,\alpha}\|^{2}&\le C_{\delta,K}\gamma^{2}\sum_{l,m=1}^{K}(\Lambda_{u,\alpha})_{m}(\Lambda_{u,\alpha})_{l}\mathrm{Tr}\big(A_{l}VA_{m}^{T}\big)\\
&\le C_{\delta,K}\gamma^{2}\bigg(\sum_{l=1}^{K}(\Lambda_{u,\alpha})_{l}\sqrt{\mathrm{Tr}(A_{l}^TVA_{l})}\bigg)^{2}\\ 
&\le \gamma^{2}C_{\delta,K}\|\Lambda_{u,\alpha}\|^{2}\sum_{l=1}^{K}\mathrm{Tr}(A_{l}^TVA_{l}),
\end{align*}
where $C_{\delta,K}$ is depending only on $\delta$ and $\ln(\frac{K}{\delta})$.
On $\mathcal{B}_{\delta},$ we have 
$$\mathrm{I}_{\alpha}\le C_{\delta,K}\gamma^{2}\|\Lambda_{u,\alpha}\|^{2}\sum_{l=1}^{K}\mathrm{Tr}\big(A_{l}^TVA_{l}\big).$$
\paragraph*{Final bound for $\mathrm{I}_{\alpha}$ in Scenario 4}
For normalized $u\in\mathbb{R}^{K},$ $\|\Lambda_{u,\alpha}\|^{2}\le \rho_{\min}(H^TH)^{-1}$.
Considering $u\in\mathbb{R}^{K}$ such that $u^Tu=1$, we have on $\mathcal{B}_{\delta},$
\begin{equation}\label{Eq:ProjOperatorS4TermIAlph}
\mathrm{I}_{\alpha}\le C_{\delta,K}\gamma^{2}\frac{\sum_{l=1}^{K}\mathrm{Tr}\big(A_{l}^TVA_{l}\big)}{\rho_{\min}(H^TH)}.\end{equation}
\subsubsection*{Scenario 4: Term $\mathrm{II}_{\alpha}$}
Using the decomposition \eqref{Eq:ProjOperatorDecompIIAlph},
$$\mathrm{II}_{\alpha} \le 4\|E^TH\Lambda_{u,\alpha}\|^2+4\|H^TE\Lambda_{u,\alpha}\|^2+2\|E^TE\Lambda_{u,\alpha}\|^2.$$
We focus on the two first terms.
Replacing $\Lambda_{u}$ by $\Lambda_{u,\alpha}$ for the term $\mathrm{II}_{\alpha}$ in Scenario 4 in \cref{Subsub:Scen4} we obtain similar results. We have, on $\mathcal{B}_{\delta},$
\begin{align*}
\|H^TE\Lambda_{u,\alpha}\|^{2}&\le C_{\delta,K}\gamma^{2}\rho(H^TH)\|\Lambda_{u,\alpha}\|^{2}\sum_{l=1}^{K}\mathrm{Tr}(A_{l}VA_{l}^{T}),\\
\|E^TH\Lambda_{u,\alpha}\|^{2}&\le C_{\delta,K}\gamma^{2}\|H\Lambda_{u,\alpha}\|^{2}\sum_{l=1}^{K}\mathrm{Tr}(A_{i}VA_{i}^{T}).
\end{align*}
For the last term we use the following inequalities,
$$\|E^TE\Lambda_{u,\alpha}\|^{2}\le \rho(E^TE)\|E\Lambda_{u,\alpha}\|^{2}\le \frac{\alpha}{2}\|E\Lambda_{u,\alpha}\|^{2}.$$
Following the same procedure we get 
$$\mathrm{II}_{\alpha}\le C_{\delta,K}\gamma^{2}\bigg(\rho(H^TH)\|\Lambda_{u,\alpha}\|^{2}+\|H\Lambda_{u,\alpha}\|^{2}+\frac{1}{2}\alpha\|\Lambda_{u,\alpha}\|^{2}\bigg)\bigg(\sum_{l=1}^{K}\mathrm{Tr}\big(A_{l}^TVA_{l}\big)\bigg).$$
\paragraph*{Final bound on $\mathrm{II}_{\alpha}$ in Scenario 4:}
Using $\rho_{\min}\left((\hat{H}^T\hat{H})_{\alpha}\right)\ge\max(\alpha,\frac{\rho_{\min}(H^TH)}{2})$, on $\mathcal{B}_{\delta},$ we obtain
\begin{equation}\label{Eq:ProjOperatorS4TermIIAlph}
\rho_{\min}\big((\hat{H}^T\hat{H})_{\alpha}\big)^{-1}\mathrm{II}_{\alpha}\le C_{\delta,K}\gamma^{2}\big(2\mathrm{Cond}(H^TH)+2+\frac{1}{2}\big)\frac{\sum_{l=1}^{K}\mathrm{Tr}\big(A_{l}^TVA_{l}\big)}{\rho_{\min}(H^TH)}.
\end{equation}
\subsubsection[]{Scenario 4: Term $\mathrm{III}_{\alpha}$}
The last quantity is $\|(H^TH)_{\alpha}^{-1/2}E^Tu\|^{2}$. It satisfies 
$$\mathrm{Var}\left((H^TH)_{\alpha}^{-1/2}E^Tu\right)=(H^TH)_{\alpha}^{-1/2}\mathrm{Var}(E^TU)(H^TH)_{\alpha}^{-1/2}.$$
Then, 
$$\mathrm{Tr}\left(\mathrm{Var}((H^TH)_{\alpha}^{-1}E^Tu)\right)=\mathrm{Tr}\big((H^TH)_{\alpha}^{-1}B\big),$$ where $B_{ij}=\gamma^{2}u^TA_{i}VA_{j}^Tu$.
Inquality \eqref{Eq:ProjOperatorTraceTrick} on the matrices $(H^TH)^{-1}$ and $B$ yields
$$\mathrm{Tr}\big((H^TH)^{-1}B\big)\le \rho_{\min}(H^TH)^{-1}\mathrm{Tr}(B).$$
We apply the deviation result of \cref{Prop:DeviationResult} in order to have, on $\mathcal{B}_{\delta}$,
$$\mathrm{III}_{\alpha}\le C_{\delta,K} \gamma^{2}\rho_{\min}\big((H^TH)_{\alpha}\big)^{-1}\|u\|^{2}\sum_{l=1}^{K}\mathrm{Tr}(A_{l}^TVA_{l}).$$
\paragraph*{Final bound for $\mathrm{III}_{\alpha}$ in Scenario 4:}
Considering $u\in\mathbb{R}^{K}$ such that $u^Tu=1,$ we have $\mathcal{B}_{\delta},$
\begin{equation}\label{Eq:ProjOperatorS4TermIIIAlph}
\mathrm{III}_{\alpha}\le C_{\delta,K}\gamma^{2}\frac{\sum_{l=1}^{K}\mathrm{Tr}\big(A_{l}^TVA_{l}\big)}{\rho_{\min}(H^TH)}.
\end{equation}
We conclude the proof of \cref{Th:ProjOperatorGeneralizedPenalinise}, thanks to the decomposition of \cref{Th:ProjOperatorSkeletonProofRegularized} and inequalities \eqref{Eq:ProjOperatorS4TermIAlph}, \eqref{Eq:ProjOperatorS4TermIIAlph} and \eqref{Eq:ProjOperatorS4TermIIIAlph}.

%%%%%%%%%%%%%%%%%%%%%%%%%%%%%%%%%%%%%%%%%%%%%%%%%%%%%%%
%%%%%%%%%%%%%%%%%%%%%%%%%%%%%%%%%%%%%%%%%%%%%%%%%%%%%%%
%% 
%%%%%%%%%%%%%%%%%%%%%%%%%%%%%%%%%%%%%%%%%%%%%%%%%%%%%%%
%%%%%%%%%%%%%%%%%%%%%%%%%%%%%%%%%%%%%%%%%%%%%%%%%%%%%%

\printbibliography

\end{document}